\documentclass[11pt]{amsart}

\usepackage[margin=2cm]{geometry}

\usepackage{palatino}
\usepackage{latexsym, amssymb, amsfonts, amsmath, amsthm, graphics, graphicx, subfigure, bbm, stmaryrd}
\usepackage[pdftex,colorlinks,linkcolor=blue,menucolor=red,backref=false,bookmarks=true]{hyperref}

\newcommand{\field}[1]{\mathbb{#1}}
\newcommand{\R}{\field{R}}

\newcommand{\C}{\field{C}}
\newcommand{\Z}{\field{Z}}
\newcommand{\p}{\field{P}}

\newcommand{\T}{\mathbf{T}}

\newcommand{\Tr}{\mathrm{Tr}}

\newcommand{\Var}{\mathrm{Var}}
\newcommand{\supp}{\mathrm{supp}}
\newtheorem{theorem}{Theorem}
\newtheorem{conjecture}{Conjecture}

\newtheorem{proposition}{Proposition}
\newtheorem{definition}{Definition}
\newtheorem{lemma}{Lemma}
\newtheorem{corollary}{Corollary}
\newtheorem{remark}{Remark}

\numberwithin{equation}{section}

\begin{document}
\title{Free Functional Inequalities on the Circle}
\author{Ionel Popescu} 
\address{School of Mathematics\\
  Georgia Institute of Technology\\
  Atlanta, GA, 30332 \\
  IMAR\\ 
21 Calea Grivitei Street, 010702 Bucharest, Romania\\
FMI \\
Str. Academiei nr.14, sector 1, C.P. 010014, Bucharest, Romania}
\email[Ionel Popescu]{ipopescu@math.gatech.edu}

\thanks{This work was partially supported by Simon Collaboration grant no. 318929 and UEFISCDI grant PN-III-P4-ID-PCE-2016-0372}
\date{}

\begin{abstract} In this paper we deal with free functional inequalities on the circle.  There are some interesting changes from their classical counterparts.  For example, the free Poincar\'e inequality has a slight change which  seems to account for the lack of invariance under rotations of the base measure.  Another instance is the modified Wasserstein distance on the circle which provides the tools for analyzing transportation, Log-Sobolev, and HWI inequalities.   

These new phenomena also indicate that they have classical counterparts, which does not seem to have been investigated thus far. 
\end{abstract}

\maketitle
\section{Introduction}

An intensive area of research nowadays is functional inequalities.  In the classical case these reflect various aspects of operator theory, mass transport, concentration of measure, isoperimetry and analysis of Markov processes.  

To describe the setup, we start with a Riemannian manifold $M$ and a  probability measure $\nu$ on it.   The classical transportation cost inequality states that for some $\rho >0$ and any other choice of probability measure $\mu$ on $M$ the following holds
\[\tag{$T(\rho)$}
\rho \, W_{2}^{2}(\mu,\nu)\le E(\mu|\nu).
\]
Here $W_{2}(\mu,\nu)$ denotes the Wasserstein distance between $\mu $ and $\nu$ (probability measures of finite second moment) given by 
\[
W_{2}(\mu,\nu)=\inf_{\pi\in\Pi(\mu,\nu)}\left( \iint d(x,y)^{2}\pi(dx,dy) \right) ^{1/2}
\]
where $\Pi(\mu,\nu)$ stands for the set of probability measures on $M\times M$ with marginals $\mu$ and $\nu$ and $d(x,y)$ is the geodesic distance between the points $x,y$.  Also we use here 
\[
E(\mu|\nu)=\int \log\frac{d\mu}{d\nu}\, d\mu
\] 
for the relative entropy of $\mu$ with respect to $\nu$ if $\mu<<\nu$ and $+\infty$ otherwise.  

The classical Log-Sobolev inequality states that for any $\mu$
\[\tag{$LSI(\rho)$}
E(\mu|\nu)\le \frac{1}{2\rho} \, I(\mu|\nu)
\]
where
\[
I(\mu|\nu)=\int\Big|\nabla \log \frac{d\mu}{d\nu} \Big|^{2}d\mu
\]
is the Fisher information of $\mu $ with respect to $\nu$ defined in the case $\mu<<\nu$ and $\frac{d\mu}{d\nu}$ being differentiable.

Another more refined inequality is the HWI inequality introduced in \cite{OV} which states that for any choice of the measure $\mu$,
\[\tag{$HWI(\rho)$}
E(\mu|\nu)\le \sqrt{I(\mu|\nu)} \, W_{2}(\mu,\nu)-\frac{\rho}{2} \,  W_{2}^{2}(\mu,\nu).
\]
Finally, the classical Poincar\'e inequality states that for any compactly supported and smooth function $\psi$
on $M$,  
\[\tag{$P(\rho)$} 
\rho \, \Var_{\nu}(\psi)\le \int |\nabla \psi|^{2}\nu (dx)
\]
where $ \Var_{\nu}(\psi) = \int \psi^2(x) \nu (dx) - \big ( \int \psi (x) \nu (dx) )\big )^2$ is the variance
of $\psi $ with respect to $\nu$.

It is known that if a measure $\nu$ on $M$ has the form 
\[
\nu(dx)=e^{-V(x)}dx,
\]
where $V:M\to\R$ satisfies the famous Bakry-\'Emery condition 
\[
\mathrm{Hess}(V)+Ric\ge \rho,
\]
then, $T(\rho)$, $LSI(\rho)$, $HWI(\rho)$ and $P(\rho)$ hold true (see for a sample of the literature \cite{OV}, \cite{Bakry}, \cite{L2}, \cite{Villani2}, \cite{Wang},\cite{BL2}, \cite{Gozlan}, or \cite{Baud2} in sub-Riemannian framework).  The Gaussian model is the case of $V(x)=|x|^{2}/2+C$ on $\R^{d}$ and the Log-Sobolev in this case is due to \cite{Gross}, while the transportation was first introduced in \cite{Tal}. 

The bridge from classical to free probability is usually done using random matrices (see for instance Voiculescu's work \cite{MR1094052}, the monograph \cite{MR1217253} and also \cite{Biane2,HPU2,HPU1}).  It was shown by Ben Arous and Guionnet in \cite{BAG1} that  one can realize the free entropy as the rate function of the large deviations for the distribution of eigenvalues of $n\times n$ complex random matrix ensembles (see also \cite{J}).   To explain the main ideas, let $V:\R\to\R$ be a smooth function with enough growth at infinity and define the probability distribution 
\[
\p_{n}(dM)=\frac{1}{Z_{n}} \, e^{-n\mathrm{Tr}_{n}(V(M))}dM
\]
on the set $\mathcal{H}_{n}$ of  complex Hermitian $n\times n$ matrices where $dM$ is the Lebesgue measure on $\mathcal{H}_{n}$.  For a matrix $M$, set $\mu_{n}(M)= \frac{1}{n}\sum_{k=1}^{n}\delta_{\lambda_{k}(M)}$  to be the distribution of eigenvalues of $M$.  These are random variables with values in $\mathcal{P}(\R)$, the set of probability measures on $\R$ and we know that they converge almost surely to a non-random measure $\mu_{V}$ on $\R$.   For a measure $\mu$ on $\R$, the logarithmic energy with external field $V$ is defined by
\[
E(\mu)=\int V(x)\mu(dx)-\iint \log|x-y| \, \mu(dx)\mu(dy).
\]
The minimizer of $E(\mu)$ over all probability measures on $\R$ is called the equilibrium measure $\mu_{V}$.  From \cite{BAG1} we learn that the distributions of $\{\mu_{n}\}_{n\ge1}$ under $\p_{n}$ satisfy a large deviations principle with scaling $n^{2}$ and rate function given by 
\[
R(\mu)=E(\mu)-E(\mu_{V}).
\]
In the case of $V(x) = x^2/2$, we have the standard Gaussian Unitary Ensemble in random matrix theory which gives rise to the semicircular law as equilibrium measure.  

Assume that $V''\ge2\rho$.  Using the classical inequalities on large random matrices, for convex potentials, one can prove the analog of the transportation inequality (due initially to Biane and Voiculescu in \cite{BV} and then \cite{HPU1}) and Log-Sobolev (see \cite{Biane2}).  A different method, still partially based on random matrices to prove a Brunn-Minkowski and from this to deduce the free transportation and Log-Sobolev is given by Ledoux in \cite{L}.   Another  more direct approach is based on a method involving the mass transportation tools which are presented in \cite{IP} and \cite{LP}.  To the point, the free transportation inequality states that for  any probability measure on $\R$, 
 \begin{equation}\label{e:i2}
 \rho \, W_{2}^{2}(\mu,\mu_{V})\le E(\mu)-E(\mu_{V}).
 \end{equation}
  The free Log-Sobolev asserts that 
 \begin{equation}\label{e:i1}
 E(\mu)-E(\mu_{V})\le \frac{1}{4\rho} \, I(\mu)
\end{equation}
for any probability measure $\mu$ on $\R$ whose density with respect to the Lebegue measure is in $L^{3}(\R)$, with the notation
\begin{equation}\label{ei:1I}
I(\mu)=\int\big (H\mu(x)-V'(x) \big )^{2}\mu(dx).
\end{equation}
Here $H\mu= 2\int \frac{1}{x-y}\mu(dx)$ stands for the Hilbert transform of $\mu$.   

The free Poincar\'e inequality was introduced in \cite{LP} and later carefully analyzed in \cite{LP2}.  It states that if the equilibrium measure $\mu_{V}$ of $V$ has support $[-2,2]$, then, 
\begin{equation}\label{e:i4}
 \frac{2\rho}{\pi^{2}}\int_{-2}^{2}\int_{-2}^{2}\left(\frac{\phi(x)-\phi(y)}{x-y}\right)^{2}\frac{4-xy}{\sqrt{4-x^{2}}\sqrt{4-y^{2}}} \, dxdy\le \int \phi'(x)^{2}\mu_{V}(dx),  
 \end{equation}
for any smooth function $\phi$ on the interval $[-2,2]$.     This is further refined in \cite{IP4} for the case of the quadratic potential $V(x)=x^{2}/2$.

The purpose of this paper is to expand the study of free functional inequalities to the circle.  Previously,  such inequalities were first introduced by Hiai, Petz and Ueda in \cite{HPU1}.  Particularly, the transportation and Log-Sobolev inequalities on the circle were proved using unitary random matrix ensembles.    We should point out a key fact which will play an important role in the sequel.   In the random matrix approach from \cite{HPU1}, the key idea is to look at the Haar measure on $U(n)$, the space of unitary matrices.  The main issue with this is that $U(n)$ has Ricci curvature constant in all directions,  except one direction in which it is $0$.  Therefore, the Bakry-\'Emery condition which is typically used to get functional inequalities does not hold in this case.  The fix is to actually look at the subgroup of matrices with determinant one.  This has the Ricci curvature constant and thus, we can apply the classical inequalities, which in the limit provides the free analogs.  This phenomena produces various corrections as we will discuss below.  

Our approach is not based on the random matrix approach, but rather on direct methods based on mass transport tools.  We recapture the free transportation and Log-Sobolev alongside with HWI and Brunn-Minkowski.  We also  introduce the free Poincar\'e inequality.  

We start the discussion on our approach introducing first the free Poincar\'e inequality (which to our knowledge is new) similar to \eqref{e:i4}.  If $\mu$ is a probability on the unit circle, we say it satisfies a free Poincar\'e inequality with constant $\rho>0$ if 
\begin{equation}\label{e:i1c}
 2\rho \iint\left| \frac{f(z)-f(w)}{z-w}  \right|^{2}\, \alpha(dz)\,\alpha(dw)\le \Var_{\mu}(f'):= \int |f'|^{2}d\mu-\left| \int f'\,d\mu\right|^2
\end{equation}
holds for any smooth function $f$ on the circle.  Here we denote by $\alpha$ the Haar measure on the circle.  Notice that there are similarities and also differences to the classical analog and the free version \eqref{e:i4}.  A difference from the classical counterpart is due to the fact that the left hand sides of \eqref{e:i4} and \eqref{e:i1c} are independent of the measure $\mu$. From a random matrix perspective the fact that the left hand side of \eqref{e:i4} and \eqref{e:i1c} do not depend on the measure $\mu$ is a reflection of universality in random matrix theory for the fluctuations (see the argument in \cite{LP} where \eqref{e:i4} is introduced). 

The other particularity of \eqref{e:i1c} is that as opposed to the \eqref{e:i4} and the classical Poincar\'e inequality, the right hand side has an extra term, namely $\left| \int f'\,d\mu\right|^2$.  In the case of $\mu=\alpha$, $\int f'd\alpha=0$ for any smooth function, thus the inequality, above takes the form 
\[
\iint\left| \frac{f(z)-f(w)}{z-w}  \right|^{2}\, \alpha(dz)\,\alpha(dw)\le  \int |f'|^{2}d\mu
\]
which has the same flavor as \eqref{e:i4}.  It is worth pointing out that the Haar measure $\alpha$ is the unique measure on the circle which is invariant under rotations and thus the unique measure $\mu$ with the property that $\int f' d\mu=0$ for any smooth function on the circle.   Therefore, we can interpret the right hand side term $\int f'd\mu$ as a measure of lack of invariance under rotation for $\mu$.   This indicates that the free Poincar\'e inequality unveils an interesting phenomena.   

Furthermore, given a potential $Q$ on the circle, we define $\mu_Q$ to be the measure which minimizes
\[
E_Q(\mu)=\int Q d\mu -\iint \log|x-y|\mu(dx)\mu(dy)
\]
over all probability measures on the circle.   In the case $Q\equiv 0$, the equilibrium measure $\mu_{Q}$  is the Haar measure $\alpha$.    We show that under the convexity assumption $Q''\ge\rho-1/2$, the equilibrium measure $\mu_{Q}$ satisfies  \eqref{e:i1c}. 

In general, the free transportation inequality with constant $\rho>0$ (which we call $T(\rho)$) associated to the potential $Q$ states that for any probability measure $\mu$ on $\T$ the following holds true
\begin{equation}\label{e:i2c}
\rho W_2^{2}(\mu,\mu_Q)\le E_{Q}(\mu)-E_{Q}(\mu_{Q}).
\end{equation}
Under the assumption $Q''\ge\rho-1/2$, the transportation inequality was introduced in \cite{HPU1} using random matrix approximations is $T(\rho/2)$.  In this paper we strengthen this inequality to the following
\begin{equation}\label{e:i2c2}
\frac{\rho}{2} \mathcal{W}_2^{2}(\mu,\mu_Q)\le E_{Q}(\mu)-E_{Q}(\mu_{Q}),
\end{equation}
where $\mathcal{W}_{2}$ is a larger quantity than the standard Wasserstein distance.  The new distance here is defined as 
\[
\mathcal{W}_{2}(\mu,\nu)=\sup\left\{ W_{2}(\bar{\mu},\bar{\nu}): \text{ such that } \int x\,\bar{\mu}(dx)=\int x\,\bar{\nu}(dx)\right\},
\]
where the supremum is taken over all choices of $u,v\in[0,2\pi)$ and measures $\bar{\mu}$ on $[u,u+2\pi)$ and $\bar{\nu}$ on $[v,v+2\pi)$ such that $\exp_{\#}\bar{\mu}=\mu$ and $\exp_{\#}\bar{\nu}=\nu$ with $\exp(t)=e^{it}$ for $t\in\R$.   This is discussed in Section \ref{S:7}.   The main comment is that this version of the Wasserstein distance reveals, to some extent, a phenomena which is similar to the one of Poincar\'e inequality, namely that the circle acts on itself by rotations.  Several of the standard transportation techniques can be carried out for this new distance.   These results are enough to show the transportation inequality \eqref{e:i2c2}.  Worth noticing is the fact that even in the case $Q\equiv 0$ (when we have $\mu_{Q}=\alpha$) the inequality \eqref{e:i2c2} is not sharp (see Proposition~\ref{p:6}).  We conjecture though  that the right coefficient in front of \eqref{e:i2c2} is $\rho$ not $\rho/2$.

Furthermore, we introduce and analyze the Log-Sobolev inequality $LSI(\rho)$ in the form (introduced in \cite{HPU1})
\begin{equation}\label{e:i3}
E_{Q}(\mu)-E_{Q}(\mu_{Q})\le \frac{1}{4\rho}I_{Q}(\mu)
\end{equation}
with 
\[
I_{Q}(\mu):=\int(H\mu-Q')^{2}d\mu-\left(\int Q'd\mu\right)^{2}\quad\text{ and }\quad H\mu(x)=-p.v. \int \frac{z+w}{z-w}\mu(dw).
\]
Again under convexity assumption $Q''\ge\rho-1/2$, we prove $LSI(\rho/2)$ holds true using the transportation tools developed for the new distance.  As in the case of transportation inequality, $LSI(\rho/2)$ is not sharp and  we conjecture that $LSI(\rho)$ should be the optimal version, at least in the case $Q\equiv0$.  

Another observation is that we see the same phenomena as in the Poincar\'e inequality, namely there is a correction in $I_{Q}(\mu)$ in comparison to \eqref{ei:1I} which has an extra term.  This seems to vanish for all measures $\mu$ if and only if $Q\equiv 0$, which is the case of $\mu_{Q}=\alpha$, the Haar measure on $\T$. 

Under the same convexity conditions, we also prove the HWI inequality which claims that
\begin{equation}\label{e:i4H}
E_{Q}(\mu)-E_{Q}(\mu_{Q})\le \sqrt{I_{Q}(\mu)}\mathcal{W}_{2}(\mu,\mu_{Q})-\frac{\rho}{2} \mathcal{W}_{2}(\mu,\mu_{Q})^{2}.
\end{equation}
It seems very difficult to get this inequality using random matrix approximations.  Again this is probably not sharp as the constant $\rho/2$ should be $\rho$, though we do not have an argument for it.  We also do not know if this holds for the standard Wasserstein distance $W_{2}$ instead of $\mathcal{W}_{2}$.    

Another result which follows very easily in this framework is the Brunn-Minkowski's inequality and is discussed in Section~\ref{S:8}.   

Another part of this paper deals with the relations between these inequalities which parallel the hierarchy of inequalities in the classical case from \cite{OV}.   

Although the convexity of the potential is important in deducing the inequalities thus far, we also treat inequalities which are potential independent.  This was first introduced in \cite{Maida} and then also discussed in \cite{IP3}.   The main transportation inequality uses the $\mathcal{W}_{1}$ distance instead of $\mathcal{W}_{2}$.   

One aspect which transpires in the free case is the particular structure of the circle, and perhaps more precisely the fact that the circle acts on itself by rotations.  This perspective can be taken back to the classical case where it seems very natural to ask about refinements of classical inequalities in the case of manifolds with an action of a Lie group.   Though we posed this question to several experts in the area, we have not been pointed to any reference in the literature.   We suggest a version of Poincar\'e inequality in this framework, though the counterpart of the Wasserstein distance is probably more intricate.   

The paper is organized as follows.  Section~\ref{S:2} contains the preliminaries.  In Section~\ref{S:3} we discuss the main results around the minimization of the logarithmic potentials and introduce a number of operators which play an important role in the rest of the paper. Following this we introduce the free Poincar\'e inequality in Section~\ref{S:4} which is then expanded in Sections~\ref{S:5} and~\ref{S:6}.  In Section~\ref{S:7} we discuss the Wasserstein distance followed by Section~\ref{S:8} which introduces the free transportation, Log-Sobolev, HWI, and Brunn-Minkowski.  The next section, namely Section~\ref{S:9}, discusses the hierarchy  between these inequalities.   Section~\ref{S:10} gives the cases of potential independent versions of the inequalities in discussion.   The last section is a return to the classical case which asks some natural questions about the functional inequalities under the presence of a Lie group acting on the underlying space.

\section{Preliminaries}\label{S:2}   

In this section we introduce some basic notions and notations.  

For any complex number $z\in\C$, $\Re(z)$ and $\Im(z)$ will stand for the real, respectively imaginary parts.  

The normalized Haar measure on $\T$  will be denoted by $\alpha$ and is given formally as the unique measure of mass $1$ which is invariant under rotations on $\T$.   A more concrete way of computing with it goes through   
$\int f(z)\alpha(dz)=\frac{1}{2\pi}\int_{-\pi}^{\pi} f(e^{ix})dx$ for any continuous function $f:\T\to\C$.

The inner product in $L^{2}(\mu)$ is denoted by $\langle\cdot,\cdot \rangle_{\mu}$ and if no measure is specified, this inner product is taken with respect to the Haar measure $\alpha$.  We also use the notation $L^{2}_{0}(\mu)$ for the space of functions with mean $0$ with respect to the measure $\mu$.   

For a smooth function on $\T$ we define 
\[
\phi'(z)=\frac{d}{dt}\big|_{t=0}\phi(ze^{it})
\]
with the obvious extension to the higher derivatives.  

The rotation invariance property of $\alpha$ is equivalent to the fact that for any $C^{1}$ function $\phi:\T\to\C$, 
\begin{equation}\label{e:p:0}
\int \phi'\,d\alpha=0.
\end{equation}
In particular for $\phi(z)=z^{n}$, $n\in\Z$, $n\ne0$, $\phi'(z)=niz^{n}$ and thus, 
\begin{equation}\label{e:p:1}
\int z^{n}\,d\alpha=0.
\end{equation}
This in turn yields that $\{z^{n}\}_{n\in \Z}$ forms an orthonormal basis in $L^{2}(\alpha)$.

We say that a measure $\mu$ on $T$ is \emph{smooth} if it has a smooth density with respect to the Haar measure $\alpha$.

A \emph{potential} on $\T$ is simply a function $Q:\T\to\R$ which we assume to be at least  continuous.  
For a probability measure $\mu$ on $\T$, the \emph{logarithmic energy with external field $Q$} is given by 
\begin{equation}\label{ep:lp}
E_{Q}(\mu)=\int Q\,d\mu-\iint \log|x-y|\mu(dx)\mu(dy).
\end{equation}
It is known that given a continuous $Q:\T\to\R$, (see \cite{ST})  there is a unique minimizer $\mu_{Q}$ on the set of probability measures on $\T$.  We will denote for simplicity $E_{Q}=E_{Q}(\mu_{Q})$.    In what follows, we denote the support of the measure $\mu$ by $\supp{\mu}$.

The variational characterization of the equilibrium measure $\mu_{Q}$ of \eqref{ep:lp} on the set $\T$ (cf.  \cite[Thm.I.1.3]{ST}) is 
\begin{equation}\label{ep:var}
\begin{split}
Q(x)&\ge 2\int \log|x-y|\mu(dy)+C \quad\text{quasi-everywhere on }\T \\ 
Q(x)&= 2\int \log|x-y|\mu(dy)+C 
\quad\text{quasi-everywhere on}\:\:\supp{\mu}. \\ 
\end{split}
\end{equation} 
This means, in particular that the equality on $\supp \mu$ is actually 
almost surely with respect to any measure of finite logarithmic energy.  

As we will see below in Corollary~\ref{c:1},  the equilibrium measure on $\T$ with no potential ($Q\equiv0$) is simply the Haar measure $\alpha$.

Very often we will move measures and functions from $\T$ to the interval $[-\pi,\pi)$ or $[0,2\pi)$ and we will do this without specifying it all the time.  As a piece of notation, we will fix along the discussion the potential $Q$ on $\T$ and set $V(x)=Q(e^{ix})$.  With this mapping, the characterization of the probability measure $\mu_{Q}$ becomes the characterization of the probability measure $\mu_{V}$ on $[-\pi,\pi)$ which is the minimizer of the energy 
\begin{equation}\label{e1:01}
E_{V}(\nu)= \int V(x)\nu(dx)-\iint\log|e^{ix}-e^{iy}|\nu(dx)\nu(dy).  
\end{equation}
The variational characterization of $\mu_{V}$ becomes 
\begin{equation}\label{ep:var2}
\begin{split}
V(x)&\ge 2\int \log|e^{ix}-e^{iy}|\mu_{V}(dy)+C \quad\text{quasi-everywhere on }\T \\ 
V(x)&= 2\int \log|e^{ix}-e^{iy}|\mu_{V}(dy)+C 
\quad\text{quasi-everywhere on}\:\:\supp{\mu_{V}}. \\ 
\end{split}
\end{equation} 

It is useful to observe that in the case $Q$ is $C^{1}$, then we have 
\begin{equation}\label{e1:03}
\int Q'\,d\mu_{Q}=0.
\end{equation}
Indeed, taking the measures $\mu_{t}=(\theta_{t})_{\#}\mu_{Q}$ for small $t$  and $\theta_{t}(z)=e^{it}z$ and noticing that 
\[
E_{Q}(\mu_{t})-E_{Q}=\int (Q(e^{it}z)-Q(z))\mu_{Q}(dz)
\]
is minimized for $t=0$, implies the derivative at $0$ is 0, thus concluding \eqref{e1:03}.   In this paper, if $\phi:X\to Y$ is a measurable map and $\nu$ is a measure on $X$, then the pushforward $\phi_{\#}\nu$ is defined as 
\[
(\phi_{\#}\nu)(A)=\nu(\phi^{-1}(A)) \text{ for any }A\subset Y.  
\]

\section{Minimization of the logarithmic energy}\label{S:3}

The main results in this section are based on the following lemma. 

\begin{lemma}
For any $z,w\in\T$ with $z\ne w$, then 
\begin{equation}\label{e1:001}
\log|z-w|=-\sum_{n\ge1}\frac{1}{n}\Re(z^{n}\bar{w}^{n}),
\end{equation}
where the series is convergent uniformly on $|z-w|\ge\epsilon$ for each small $\epsilon>0$.  

Also for any finite measure $\mu$ on $\T$,
\begin{equation}\label{el:1}
\int\log|z-w|\mu(dw)=-\sum_{n\ge1}\frac{1}{n} \Re\left(\left(\int w^{-n}\mu(dw) \right)z^{n}\right)
\end{equation}
with the meaning that either side is finite if and only if the other one is in which case, both quantities are equal. Moreover, if one of the sides is $-\infty$, then the other is also $-\infty$.  

Finally, 
\begin{equation}\label{e1:003}
\iint \log|z-w|\mu(dz)\mu(dw)=-\sum_{n\ge1}\frac{1}{n}\left| \int z^{n}\mu(dz) \right|^{2}
\end{equation}
which means that the one side is finite if and only if the other side is also finite.   
\end{lemma}

\begin{proof}
We will work with the branch of the complex logarithm which is defined on the plane without the positive line.  For a complex number $z$ with $|z|<1$, we have the standard expansion
\[
\log(1-z)=-\sum_{n\ge1}\frac{z^{n}}{n}.
\]

By replacing $w$ by $\bar{z}w$, it suffices to show \eqref{e1:001} and \eqref{el:1} for the case $z=1$.  Now, for $w\in\T$ and $0<r<1$  such that $r>1-|1-w|^{2}/2$ it is easy to check that
\begin{equation}\label{e1:002}
\log|1-rw|\le \log|1-w| \le  \log|1-w/r|=-\log(r)+\log|1-rw|.  
\end{equation}
Next we argue that 
\[
\log|1-rw|=\Re(\log(1-rw))=-\sum_{n\ge1}\frac{r^{n}w^{n}}{n}.  
\]
It is clear that we can sandwich the $\log|1-w|$, however we need to make sure the series $-\sum_{n\ge1}\frac{r^{n}w^{n}}{n}$ converges to $-\sum_{n\ge1}\frac{w^{n}}{n}$ as $r\nearrow1$.  This is relatively standard in Fourier series, though a quick argument is based on the following elementary formula valid for $0\le r\le 1$: 
\begin{equation}\label{e1:004}
\sum_{n=1}^{N}\frac{r^{n}w^{n}}{n}=\sum_{n=1}^{N}\int_{0}^{1}r^{n}w^{n}t^{n-1}dt=\int_{0}^{1}\frac{rw(1-r^{N}w^{N}t^{N})}{1-rwt}dt.
\end{equation}
Dominated convergence theorem justifies that we can let $N\to\infty$ to obtain for $0\le r\le 1$,  that 
\[
\sum_{n=1}^{\infty}\frac{r^{n}w^{n}}{n}=\int_{0}^{1}\frac{rw}{1-rwt}dt.
\]
Consequently, we actually obtain with little effort that for $w\in T$, $w\ne 1$ and $r\in[\max(0,1-\frac{|1-w|^{2}}{2}),1)$,
\[
-\frac{4(1-r)}{|1-w|^{2}}\le \Re(\sum_{n=1}^{\infty}\frac{w^{n}}{n}-\sum_{n=1}^{\infty}\frac{r^{n}w^{n}}{n})\le \frac{4(1-r)}{|1-w|^{2}},
\]
which combined with \eqref{e1:002}, yields  for $w\ne 1$, 
\[
-\frac{4(1-r)}{|1-w|^{2}}-\sum_{n=1}^{\infty}\frac{\Re(w^{n})}{n}\le \log|1-w|\le -\log(r)+\frac{4(1-r)}{|1-w|^{2}}-\sum_{n=1}^{\infty}\frac{\Re(w^{n})}{n},
\]
valid for $r\in[\max(0,1-\frac{|1-w|^{2}}{2}),1)$.  Letting $r\nearrow 1$ gives \eqref{e1:001}.  

Now we argue towards \eqref{el:1}.  We claim first that  $\int \log|1-w|\mu(dw)$ is finite if and only if 
\[
\lim_{r\nearrow1}\int\log|1-rw|\mu(dw)\text{ is finite}.
\]
To see this we argue first that $\log|1-w|$ is bounded above by $\log 2$ and thus  Fatou's Lemma implies that 
\[
0\le\int \left(-\log(|1-w|/2)\right)\mu(dw)= \int\liminf_{r\nearrow1}\left(-\log(|1-rw|/2)\right)\mu(dw)\le \liminf_{r\nearrow 1}\int\left(-\log(|1-rw|/2)\right)\mu(dw)<\infty.  
\]
which guarantees the integrability of $\log|1-w|$ against the measure $\mu$.  Now finiteness of $\int\log|1-w|\mu(dw)$ combined with the inequality
\[
\log|1-w|\le \log|1-w/r|=-\log r+\log|1-rw|
\]
yields
\[
0\le -\log(|1-rw|/2)\le -\log(|1-w|/2)-\log r
\]
and from this the dominated convergence theorem takes care of the claim.  

Next, since 
\[
\int\log|1-rw|\mu(dw)=-\sum_{n\ge1}\frac{1}{n} r^{n}\Re\left(\int w^{n}\mu(dw) \right)
\]
convergence of the left hand side as $r\nearrow 1$, is equivalent to the Abel-Poisson summability of the series $-\sum_{n\ge1}\frac{1}{n}\Re\left(\int w^{n}\mu(dw) \right)$.  From the Tauberian theorem (\cite{littlewood} or \cite[Theorem 9.5.4]{Stratila}), this is equivalent to the summability of the series $-\sum_{n\ge1}\frac{1}{n}\Re\left(\int w^{n}\mu(dw) \right)$.  

For the proof of \eqref{e1:003} we use a similar argument to the proof of \eqref{el:1} replacing $\log|z-w|$ by $\log|z-rw|$ with $0<r<1$.   The details are left to the reader.  \qedhere
\end{proof}

The first application is the following.  

\begin{corollary}\label{c:1}
For the Haar measure $\alpha$ on $\T$,  
\[
\int \log|z-w|\alpha(dw)=0,\quad \forall z\in\T.
\]
Moreover, if $\mu$ is a finite measure on $\T$ such that 
\[
\int \log|z-w|\mu(dw)=c,\quad \text{ almost everywhere for all }z\in \T,
\]
where ``almost everywhere'' is with respect to the Haar measure, then $\mu=k\alpha$ for some constant $k\in \R$.  Also we must have then $c=0$.  
\end{corollary}

The only thing we have to point out here is the choice of the constant $k$ which is chosen so that $\mu(\T)=k$.  Then $\int w^{n}\mu(dw)=\int w^{n}k\alpha(dw)$ for all $n\in\Z$.  This is enough to conclude that $\mu=k\alpha$.  

With this at hand we can continue with the first result, which is a way of solving for the equilibrium measure of a given potential.  

\begin{theorem}\label{t:1}
Assume that $Q:\T\to\R$  is a $C^{3}$ potential and $A\in\R$ is a constant.  
Then,  there exists a unique signed measure $\mu$ on $\T$ of finite total variation which solves 
\begin{equation}\label{e11:0}
\begin{cases}
2\int \log|z-w|\mu(dw)=Q(z)+C \text{ almost everywhere for } z\in\T,\\
\mu(\T)=A,
\end{cases}
\end{equation}
where ``almost everywhere'' is with respect to $\alpha$. The solution $\mu$ is given by $\mu(dz)=u(z)\alpha(dz)$ where
\begin{equation}\label{e11:3}
u(z)=A+\int \frac{i(Q'(z)-Q'(w))(z+w)}{(z-w)}\alpha(dw)=A+2\int \frac{(Q'(z)-Q'(w))\Im(zw^{-1})}{|z-w|^{2}}\alpha(dw).
\end{equation}
In addition, the constant $C$ must be given by  $C=-\int Q\,d\alpha$.

Moreover, for any $C^{1}$ function $\phi$ on $\T$,
\begin{equation}\label{e11:4}
\int \phi(x)\mu(dx)=A\int\phi \,d\alpha
-\iint
\frac{(Q(z)-Q(w))(\phi(z)-\phi(w))}{|z-w|^{2}}\,\alpha(dz)\alpha(dw).
\end{equation}
\end{theorem}

\begin{proof} The uniqueness follows easily from Corollary~\ref{c:1}, and thus we only have to deal with existence.   

Now, write $Q$ as the power series
\[
Q(z)=\sum_{n\in\Z} \left(\int Q(w)w^{-n}\alpha(dw)\right)z^{n}. 
\]  
Notice that because $Q$ is $C^{3}$, integration by parts shows that the coefficients $\int Q(w)w^{-n}\alpha(dw)$ decay as $1/n^{3}$ for large $n$, therefore the series above is absolutely convergent.   

Next, we look for a solution of \eqref{e11:0} in the form $\mu=u\,d\alpha$.   With the help of \eqref{el:1}, this means that
\[
-\sum_{n\in\Z\backslash\{0\}}\frac{1}{|n|} \left(\int w^{-n}u(w)\alpha(dw) \right)z^{n}=C+\int Q\,d\alpha+\sum_{n\in\Z\backslash\{0\}} \left(\int Q(w)w^{-n}\alpha(dw)\right)z^{n}
\]
from which, equating the coefficients of $z^{n}$ for $n\in\Z$, we first get that $C=-\int Q\,d\alpha$ and
\[
\int w^{-n}u(w)\,\alpha(dw)=-|n|\int w^{-n}Q(w)\,\alpha(dw)=i\mathrm{sign}(n)\int w^{-n}(Q'(z)-Q'(w))\,\alpha(dw)\quad \forall n\ne 0.
\]
Now we recapture the function $u$ in terms of its Fourier coefficients, as 
\begin{align*}
u(z)&=A+i\sum_{n\in\Z\backslash\{0\}} \mathrm{sign}(n)\left(\int w^{-n}(Q'(z)-Q'(w))\alpha(dw)\right)z^{n} \\
&= A-2\sum_{n\ge1}\Im\left[\left(\int w^{-n}(Q'(z)-Q'(w))\alpha(dw)\right)z^{n}\right] \\ 
&= A+i\left(\int \frac{(Q'(z)-Q'(w))(z+w)}{(z-w)}\alpha(dw)\right),
\end{align*}
which yields \eqref{e11:3}.  Notice that we used the decay of $\int Q(w)w^{-n}\alpha(dw)$ to guarantee the convergence of the series above as well as the exchange of the integration with the sum.   For example the $C^{3}$ regularity of $Q$ gives that the coefficients $\int w^{-n}(Q'(z)-Q'(w))\alpha(dw)$ decay as $1/n^{2}$, thus the series 
\[
\sum_{n\ge1}\left(\int w^{-n}(Q'(z)-Q'(w))\alpha(dw)\right)z^{n}
\] is absolutely convergent.  

To get \eqref{e11:4}, we observe that it is enough to prove it for functions $\phi$ which are $C^{2}$, an easy approximation taking care of the rest.  Thus, we expand $\phi$ as a Fourier series and then write 
\begin{align*}
\int \phi\, d\mu&=\int \phi u\,d\alpha=\sum_{n\in\Z}\left(\int \phi(z)z^{-n}\,\alpha(dz)\right)\left(\int u(w)w^{n}\,\alpha(dw)\right)\\
&=A\int\phi\,d\alpha-\sum_{n\in\Z\backslash\{0\}}|n|\iint w^{n}z^{-n}Q(w)\phi(z)\,\alpha(dw)\alpha(dz)\\
&=A\int\phi\,d\alpha+\sum_{n\ge1}n\iint w^{n}z^{-n}(Q(z)-Q(w))(\phi(z)-\phi(w))\,\alpha(dw)\alpha(dz)\\
&=A\int\phi\,d\alpha+\iint \left(\sum_{n\ge1}nw^{n}z^{-n}\right)(Q(z)-Q(w))(\phi(z)-\phi(w))\,\alpha(dw)\alpha(dz) \\
=&A\int\phi\,d\alpha+\iint \frac{(Q(z)-Q(w))(\phi(z)-\phi(w))wz^{-1}}{(1-wz^{-1})^{2}}\,\alpha(dw)\alpha(dz) 
\end{align*}
which combined with the easily checked identity $\frac{wz^{-1}}{(1-wz^{-1})^{2}}=-\frac{1}{|z-w|^{2}}$ satisfied for distinct $z,w\in\T$ we conclude \eqref{e11:4}.   As technical points, notice that the equality in the first line is justified by the fact that $\phi$ is $C^{2}$, thus the series expansion of $\phi$ is absolutely convergent as the coefficients $\int \phi(z)z^{-n}\,\alpha(dz)$ decay as $1/n^{2}$.  Another point is that that the third and fourth lines hold true because both $Q$ and $\phi$ are $C^{2}$.  \qedhere

\end{proof}

One particularly interesting consequence of the above result is that under some convexity assumption on $Q$, we can conclude full support condition.  This property is well known for the case of measures on the real line, which states that if the potential is convex, then the support is one interval.    More precisely, in this framework we have the following.   

\begin{corollary}\label{c:2}
Assume $Q$ is a $C^{3}$ potential on $\T$ with $Q''(z)\ge\beta-\frac{1}{2\log 2}$ for some constant  $\beta>0$ and any $z\in\T$.  Then $\mu_{Q}$ has a continuous and positive density with respect to $\alpha$ and 
\[
\frac{d\mu}{d\alpha}\ge 2\beta\log 2 .
\]  

\end{corollary}

\begin{proof}

Consider the argument function of a complex number, $\mathrm{Arg}(z)$ to be the angle between the vector $z$ and the positive real axis, thus defined in $(-\pi,\pi]$.   Now, the convexity condition on $Q$ means that $t\to Q(we^{it})-(\beta-\frac{1}{2\log 2})t^{2}/2$ is a convex function on $(-\pi,\pi]$.  Consequently, 
\[
\frac{Q'(we^{it})-Q'(w)}{t}\ge\beta-\frac{1}{2\log 2}.  
\]
Now for another complex number $z\in\T$, we can take $t=\mathrm{Arg}(zw^{-1})$ and thus obtain 
\[
\frac{Q'(z)-Q'(w)}{\mathrm{Arg}(zw^{-1})}\ge\beta-\frac{1}{2\log 2}.  
\]
From our choice of the argument, $\Im(z)$ and $\mathrm{Arg}(z)$ have the same sign and then 
\[
\frac{(Q'(z)-Q'(w))\Im(zw^{-1})}{|z-w|^{2}}\ge\left(\beta-\frac{1}{2\log 2}\right)\frac{\mathrm{Arg}(zw^{-1})\Im(zw^{-1})}{|z-w|^{2}}=\left(\beta-\frac{1}{2\log 2}\right)\frac{\mathrm{Arg}(zw^{-1})\Im(zw^{-1})}{|1-zw^{-1}|^{2}}.
\]
  
Next, solving \eqref{e11:0} with $A=1$ gives the solution $u$ from \eqref{e11:3} and using the above inequality leads to
\[
\begin{split}
u(z)&\ge 1+2\left(\beta-\frac{1}{2\log 2}\right)\int \frac{\mathrm{Arg}(zw^{-1})\Im(zw^{-1})}{|1-zw^{-1}|^{2}}\alpha(dw)=1+\left(\beta-\frac{1}{2\log 2}\right)\frac{1}{2\pi}\int_{0}^{2\pi}\frac{t\cos(t/2)}{\sin(t/2)}dt\\ 
&=1+2\left(\beta-\frac{1}{2\log 2}\right)\frac{1}{2\pi}\int_{0}^{\pi} t\cot(t/2)dt=2\beta\log 2>0.  
\end{split}
\]
The continuity of $u$ is clear from \eqref{e11:3}.  \qedhere

\end{proof}

We introduce now the operators which play an important role in the rest of the paper.   

\begin{definition}\label{d:1} For a $C^{2}$ function $\phi:\T\to\C$, set
\begin{equation}\label{e:EN}
\begin{split}
(\mathcal{E}\phi)(z)&=-2\int \log|z-w|\phi(w)\alpha(dw), \\
(\mathcal{N}\phi)(z)&= \int \frac{(\phi'(z)-\phi'(w))(z+w)}{i(z-w)}\alpha(dw)=-2\int \frac{(\phi'(z)-\phi'(w))\Im(zw^{-1})}{|z-w|^{2}}\alpha(dw) \\
(\mathcal{U}\phi)(z)&= \int  \frac{(\phi(z)-\phi(w))(1-zw) }{(z-w)}\alpha(dw) + \phi(z)(z+1/z)\\
(\mathcal{V}\phi)(z)&= \int  \frac{(\phi(z)-\phi(w))(1-zw) }{(z-w)}\alpha(dw).
\end{split}
\end{equation}
\end{definition}

For a $C^{3}$ function $\phi$ on $\T$, Theorem~\ref{t:1} states that  $\mathcal{N}\phi$  is the unique solution $\psi$  to 
\begin{equation}\label{e1:0}
\begin{cases}
2\int \log|z-w|\psi(w)\, \alpha(dw)=-\phi(z)+\int \phi\, d\alpha \text{ almost everywhere for } z\in \T,\\
\int \psi\, d\alpha=0.
\end{cases}
\end{equation}

Here we collect a number of key properties of these operators.

\begin{proposition}\label{p:1}  
\begin{enumerate}
\item For any $C^{2}$ function $\phi $ on $\T$, $\mathcal{E}\phi$ is a $C^{2}$ function and $\mathcal{N}\phi$ is a continuous function. In addition,  if $\phi$ is $C^{3}$ then
\begin{equation}\label{e:ep:EN}
\begin{split}
\mathcal{E}\mathcal{N}\phi &= \phi - \int \phi\, d\alpha , \\
\mathcal{N}\mathcal{E}\phi &=\phi-\int \phi\, d\alpha.
\end{split}
\end{equation}

\item $\mathcal{E}1=0$ and  $\mathcal{E}z^{n}=\frac{1}{|n|}z^{n}$, $n\ne0$. For any $n\in\Z$, $\mathcal{N}z^{n}=|n| z^{n}$.

\item For any  $C^{2}$ functions $\phi,\psi:\T\to\C$,  
\begin{equation}\label{ep:9}
\langle \mathcal{N}\phi,\psi \rangle=\iint
\frac{(\phi(z)-\phi(w))(\overline{\psi(z)-\psi(w)})}{|z-w|^{2}} \, \alpha(dz)\,\alpha(dw).
\end{equation}
In particular, $\langle \mathcal{N}\phi,\psi \rangle=\langle \phi,\mathcal{N}\psi \rangle$ and $\langle\mathcal{N}\phi,\phi \rangle\ge0$.

\item $\mathcal{E}$ can be extended to a bounded non-negative operator on $L^{2}_{0}(\alpha)$ with 
\begin{equation}\label{ep:10:e}
\mathcal{E}\le I.
\end{equation}
Also $\mathcal{N}$ has a (unique) extension to a non-negative  self-adjoint operator (still called $\mathcal{N}$) such that restricted to  $L^{2}_{0}(\alpha)$ satisfies
\begin{equation}\label{ep:10:n}
I\le \mathcal{N}.
\end{equation}

\item The operators $\mathcal{U}$ and $\mathcal{V}$ are characterized by 
\begin{equation}\label{ep:uv}
\mathcal{U}z^n=\begin{cases}
z^{n+1}, &n\ge1 \\
z+z^{-1}, &  n=0 \\ 
z^{n-1}, & n\le -1. 
\end{cases}
\text{ and }
\mathcal{V}z^n=\begin{cases}
z^{n-1}, & \text{ for } n\ge1 \\
0, & n=0 \\ 
z^{n+1}, & n\le -1 
\end{cases}
\end{equation}
In addition, one can extend these operators to bounded operators on $L^{2}(\alpha)$ which satisfy $\mathcal{U}^{*}=\mathcal{V}$, $\mathcal{V}^{*}=\mathcal{U}$.  Also, $\mathcal{VU}=I$ while $\mathcal{UV}=I-\Pi$, where $\Pi$ is the projection on the set of constant functions.  

\item We have 
\[
\mathcal{VNU}=\mathcal{N}+I.
\]

\item Set $\mathcal{L}\phi=\mathcal{N}^{2}\phi$ for $C^{2}$ functions.  Then 
\begin{equation}\label{ep:8}
\langle \mathcal{L}\phi,\psi\rangle=\int \phi' \bar{\psi}'d\alpha,
\end{equation}
and
\[
(\mathcal{L}\phi)(z)=-\phi''(z)
\]
which is the negative Laplacian on $\T$.

\end{enumerate}
\end{proposition}

\begin{proof}

\begin{enumerate}
\item We prove first that if $\phi$ is $C^{2}$, then $\mathcal{E}\phi$ is also $C^{2}$.  To this end, write $z=e^{it}$ and $w=e^{is}$ and $\psi(s)=\phi(e^{is})$.  With these notations, and the identification of $\T$ with $[t-\pi,t+\pi)$ we write
\[
\begin{split}
(\mathcal{E}\phi)(z)&=-
\frac{1}{\pi}\int_{t-\pi}^{t+\pi} \log|e^{it}-e^{is}|(\psi(s)-\psi(t)-(s-t)\psi'(t))ds \\
&=-
\frac{1}{\pi}\int_{-\pi}^{\pi} \log(2|\sin(u/2)|)(\psi(u+t)-\psi(t)-u\psi'(t))du
\end{split}
\]
where we used the fact that $\int_{-\pi}^{\pi}\log(2|\sin(u/2)|)ds=0$ and the similar integral $\int_{-\pi}^{\pi}u\log(2|\sin(u/2)|)du=0$.   This expression of $\mathcal{E}\phi$ combined with the dominated convergence theorem, implies that we can take two derivatives and the second derivative is actually continuous.  The details are straightforward.  

For the operator $\mathcal{N}$, we need to show that if $\phi$ is $C^{2}$, then $\mathcal{N}\phi$ is continuous.  This follows in a similar vein from the writing 
\[
(\mathcal{N}\phi)(e^{it})=-\int_{t-\pi}^{t+\pi}\frac{\psi'(t)-\psi'(s)}{\tan((t-s)/2)}ds=\int_{-\pi}^{\pi}\frac{\psi'(t)-\psi'(t+u)}{\tan(u/2)}du=\int_{-\pi}^{\pi}\int_{0}^{1}\frac{u\psi''(t+\tau u)}{\tan(u/2)}d\tau du
\]
and the continuity follows from the dominated convergence theorem.  

The equality in \eqref{e:ep:EN} follows from Theorem~\ref{t:1}.   

\item This is a direct calculation which can be also seen from \eqref{el:1} combined with \eqref{e:ep:EN}. 

\item This follows from Theorem~\ref{t:1}, more precisely from \eqref{e11:4}.

\item The claim for the operator $\mathcal{E}$ follows from the previous item and the fact that for any $C^{3}$ function $\phi$, we can write $\phi=\sum_{n\in\Z}\gamma_{n}z^{n}$ (in $L^{2}(\alpha)$ which is also absolutely convergent) and thus 
\[
\langle \mathcal{E}\phi,\phi \rangle=\sum_{n\in\Z\backslash\{0\}}\frac{\gamma_{n}^{2}}{|n|}\le \sum_{n\in\Z}\gamma_{n}^{2}=\|\phi\|^{2}
\]  
which, from the density of $C^{3}$ functions in $L^{2}(\alpha)$ shows that one can extend the operator to a bounded operator on $L^{2}(\alpha)$.  Also, the positivity and the contraction property follows easily. 

The claim about the operator $\mathcal{N}$ can be deduced from the fact that $\mathcal{N}$ is the inverse operator of $\mathcal{E}$ on $C^{3}$ functions.  Now $\mathcal{E}$ has the extension to a positive operator on $L^{2}_{0}(\alpha)$ and thus we can use functional calculus to define the operator $\mathcal{N}$ as the inverse of $\mathcal{E}$ on $L^{2}_{0}(\alpha)$.  Clearly this is a positive operator and is an extension of the operator $\mathcal{N}$ from $C^{3}$ functions.   

The uniqueness of the extension can be argued from the fact that $L^{2}(\alpha)$ is isometric to $\ell^{2}(\Z):=\{e:\Z\to\C:\sum_{n\in\Z}|e(n)|^{2}<\infty\}$ via the isometry $\Phi(z^{n})=e_{n}$ where $e_{n}(m)=\delta_{nm}$ (it is $1$ at $m=n$ and $0$ otherwise).   In this isometry, the operator $\mathcal{N}$ becomes the multiplication by the function $f$ which is defined as $f(n)=|n|$ and this multiplication operator has a unique self-adjoint extension, thus the conclusion.

\item Equation \eqref{ep:uv} is easy to deduce from a direct calculation.  All the other statements follow easily from \eqref{ep:uv}.   

\item[(6,7)] The other properties follow from a simple check on monomial functions $z^{n}$ and then using a density argument.    \qedhere
\end{enumerate}

\end{proof}

The next result gives the solution to \eqref{e11:0} in terms of the operator $\mathcal{N}$.   

\begin{proposition} 
For a $C^{3}$ potential $Q$ on $\T$, the probability measure solution $\mu_{Q}$ of \eqref{e11:0} is written as 
\begin{equation}\label{e:d}
\mu_{Q}=\left(1-\mathcal{N}Q\right)\alpha.
\end{equation}
 If in addition, the minimizer of $E_{Q}$ on $\T$ is $\mu_{Q}$ and has full support, then 
\begin{equation}\label{e:rEV}
E_{Q}=\int Q\,d\alpha - \frac{1}{2}\langle \mathcal{N}Q,Q \rangle=\int Q\,d\alpha-\frac{1}{2}\iint 
\left|\frac{Q(z)-Q(w)}{z-w}\right|^{2}\,\alpha(dz)\,\alpha(dw).
\end{equation}

\end{proposition}

\begin{proof}
The only thing we need to prove here is the last item which follows from 
\[
E_{Q}=\int Qd\mu_{Q}-\iint \log|z-w|\mu_{Q}(dz)\mu_{Q}(dw),
\]
where $\mu_{Q}=\left(1-\mathcal{N}Q\right)\,\alpha$. Furthermore, 
\[
\begin{split}
E_{Q}&=\int Q\,d\alpha -\langle\mathcal{N}Q,Q \rangle +\frac{1}{2}\Big \langle \mathcal{E}\Big(1-\mathcal{N}Q\Big),\Big(1-\mathcal{N}Q\Big) \Big\rangle \\ 
&=  \int Q\,d\alpha -\langle\mathcal{N}Q,Q \rangle -\frac{1}{2} \Big\langle 
     \Big(Q -\int Qd\alpha \Big), \Big(1-\mathcal{N}Q \Big) \Big \rangle \\
&= \int Q\,d\alpha -\frac{1}{2} \, \langle \mathcal{N}Q,Q\rangle 
\end{split}
\]
which combined with \eqref{ep:9} gives \eqref{e:rEV}

\end{proof}

We now discuss the behavior property of $E_{Q}$, when the potential $Q$ is perturbed.  

\begin{theorem}\label{t:2}
Assume that $Q$ is of class $C^{3}$ and $\mu_{Q}=u\,\alpha$ with $\inf_{z\in\T}u(z)>0$.   Then for any $C^{3}$ functions $f,g:\T\to\R$, 
\begin{equation}\label{ep:100}
\mu_{Q+tf+t^{2}g}=\mu_{Q}-t\mathcal{N}f\,\alpha-t^{2}\mathcal{N}g\,\alpha.
\end{equation}
taken in the sense of distributions.  In addition,
\begin{equation}\label{ep:101}
E_{Q+tf+t^{2}g}=E_{Q}+t\int f\,d\mu_{Q}+t^{2}\left(\int g\,d\mu_{Q} - \frac{1}{2}\langle\mathcal{N}f,f\rangle \right)+o(t^{2}).
\end{equation}
\end{theorem}

\begin{proof}
Since, $u(z)=1-\mathcal{N}Q(z)>0$ for all $z\in\T$ and the fact that $u$ is a continuous function (see Proposition~\ref{p:1} (1)), it follows that for small $t$ and all $z\in\T$, it is also true that 
\[
1-\mathcal{N}Q(z)-t\mathcal{N}f(z)-t^{2}\mathcal{N}g(z)>0
\]
for all $z\in\T$.  Now looking at the densities as described by \eqref{e:d}, it is easy to show \eqref{ep:100}.   

For the second part, use the representation from \eqref{e:rEV} and straightforward computations combined with \eqref{e:d}.  \qedhere
\end{proof}

\section{Poincar\'e Inequality}\label{S:4}

\subsection{General Properties}

Now we have introduced all ingredients for the statement and elementary properties of the Poincar\'e's inequality.  We start with the definition first.  

\begin{definition}
We say that a probability 
measure $\mu$ on  $\T$ satisfies the free Poincar\'e inequality with constant $\rho>0$, denoted $P (\rho)$, if
\begin{equation}\label{e1:1}
 2\rho \iint\left| \frac{f(z)-f(w)}{z-w}  \right|^{2}\, \alpha(dz)\,\alpha(dw)\le \Var_{\mu}(f'):= \int |f'|^{2}d\mu-\left| \int f'\,d\mu\right|^2
\end{equation}
holds for any smooth $f:\T\to\C$.
\end{definition}

The expression on the left-hand side appears also in \cite{MR1828463} in relation to fluctuations of functions of random matrices on the unitary group.  One can argue as in \cite{LP} that using Poincar\'e's inequality on large matrices, one can deduce some version of the free inequality for some sufficiently regular functions.  

Notice that, as opposed to the classical Poincar\'e where the left-hand side depends on the measure $\mu$ through its variance, in this case the left-hand side is independent of the measure $\mu$.  Also in the classical case the left hand side is the $L^{2}$ norm of the projection onto orthogonal to constants, while here this is some sort of $L^{2}$ norm of the non-commutative derivative $\frac{f(z)-f(w)}{z-w}$.  

In the case of $\mu=\alpha$, since $\int f'\,d\alpha=0$ for any smooth $f$, the free Poincar\'e inequality is equivalent to 
\[
2\rho \iint\left| \frac{f(z)-f(w)}{z-w}  \right|^{2}\, \alpha(dz)\,\alpha(dw)\le  \int |f'|^{2}d\alpha.
\]

A few elementary properties are collected in the following statement.  

\begin{proposition}\label{p:2}  Assume $\mu$ satisfies $P(\rho)$.   The following are true.  
\begin{enumerate}

\item $\mu$ has support $\T$. Moreover, if $d\mu=w\,dz$, with $w\in C^{2}(\T)$, then $ w(z)>0$ for all $z\in\T$.  
\item The constant $\rho$ in \eqref{e1:1} satisfies 
\[ 
\rho\le \frac{1}{2}
\] 
with equality if and only if $\mu=\alpha$. 

\item For  any $C^{1}$ function $f:\T\to\C$, 
\begin{equation}\label{e1:3}
\Var_{\alpha}(f)\le  \iint\left|\frac{f(z)-f(w)}{z-w}  \right|^{2}\,\alpha(dz)\,\alpha(dw),
\end{equation}
with $\Var_{\nu}(f)=\int |f|^{2}\,d\nu-\left|\int f\,d\nu\right|^{2}$ for any measure $\nu$.   
In fact, this inequality is equivalent to  $P(1/2)$ for the Haar measure $\alpha$ with equality in \eqref{e1:3} or \eqref{e1:1} being attained only for $f(z)=az$, $a\in\C$.  

\item If $\mu=w\,\alpha$, with $w\ge \rho$ on $\T$, then $\mu$ satisfies $P(\frac{\rho^2}{2})$.  
\end{enumerate}
\end{proposition}

\begin{proof}
\begin{enumerate}
\item Assume by contradiction that there is an open set $J\subset\T$ such that $\mu(J)=0$.  Take now, an open set $J_{2}$ such that $\bar{J}_{2}\subset J$ and  pick a smooth function $0\le f\le 1$ which is $0$ outside $J$ and 1 on $J_{2}$.   Apply now Poincar\'e to get a contradiction.  

\item The argument here follows the same line as in \cite{Biane2}.   Apply \eqref{e1:1} for $f(z)=z$ to show that $2\rho\le1$.  Now, if $\rho=1/2$, then, apply \eqref{e1:1} for $f=g+rz$ and thus conclude that 
\[
\begin{split}
\iint\left| \frac{g(z)-g(w)}{z-w}  \right|^{2}\,\alpha(dz)\,\alpha(dw)& +2r\iint \Re\left(\frac{g(z)-g(w)}{z-w}\right)\,\alpha(dz)\,\alpha(dw)+r^{2}  \\ 
&\le \int |g'|^{2} d\mu-2r\int \Re(ig'(z)\bar{z})\mu(dz)+r^{2}.
\end{split}
\] 
Since this is true for any $r\in\R$, it yields 
\[
\iint \Re\left(\frac{g(z)-g(w)}{z-w}\right)\,\alpha(dz)\,\alpha(dw)=-\int \Re(ig'(z)\bar{z})\mu(dz)
\]
which in turn gives 
\[
\iint \frac{g(z)-g(w)}{z-w}\,\alpha(dz)\,\alpha(dw)=-\int ig'(z)\bar{z}\mu(dz).
\]
Now, using the operator $\mathcal{N}$, the left hand side is precisely $\frac{1}{2}\langle \mathcal{N}g,z\rangle=\frac{1}{2}\langle \mathcal{N}g,\mathcal{N}z\rangle=\frac{1}{2}\langle\mathcal{L}g,z \rangle=-\int ig'(z)\bar{z}\,\alpha(dz)$, and thus we can reinterpreted the equality above as $\int g'(z)\bar{z}\,\mu(dz)=\int g'(z)\bar{z}\,\alpha(dz)$ for any smooth function $g$, which is equivalent to $\mu=\alpha$.
\item This is a simple application of the fact that the Poincar\'e in this case is translated into $\mathcal{N}\le\mathcal{N}^{2}$, which is equivalent to the spectral gap of $\mathcal{N}$, a consequence of \eqref{ep:10:n} of Proposition \ref{p:1}.  
\item This is immediate from the following sequence of inequalitites
\[
\begin{split}
\iint &\left| \frac{f(z)-f(w)}{z-w}\right|^{2}\, \alpha(dz)\,\alpha(dw)\le \int |f'(z)|^2\alpha(dz) = \frac{1}{2}\iint |f'(z)-f'(w)|^{2}\alpha(dz)\alpha(dw)\\ 
&\le \frac{1}{2\rho^2}\iint |f'(z)-f'(w)|^2\mu(dz)\mu(dw)=\frac{1}{\rho^2}\Var_{\mu}(f')
\end{split}
\]
which is $P(\rho^2/2)$ for $\mu$. \qedhere
\end{enumerate}

\end{proof}

\section{Houdr\'e-Kagan type refinements} \label{S:5}   

In this section we want to give some refinements for the Poincar\'e inequality in the case of the Haar measure $\alpha$.  To do this, we use the idea from \cite{ioBrasc}.  The starting point is the following statement.  

\begin{proposition} For any $C^{3}$ functions $\phi,\psi$ on $\T$ it holds that
\begin{equation}\label{e:pe:0}
\langle \mathcal{N}\phi,\psi \rangle = \langle \mathcal{E}\phi',\psi' \rangle .
\end{equation}
\end{proposition}

\begin{proof}
It suffices to check this for the case of $\phi(z)=z^{n}$ and $\psi(z)=z^{m}$ with $n,m$ an integer numbers.  If $n=0$, both sides are $0$ and the equality is trivial.  If $n\ne0$, then, since $\phi'=in\phi$ and $\psi'=im \psi$ combined with $\mathcal{N}\phi=|n|\phi$ and $\mathcal{E}\phi=\frac{1}{|n|}\phi$, the identity is immediately checked. \qedhere 
\end{proof}

In particular, since $\mathcal{E}$ is a contraction on $L^{2}(\T)$, we obtain another proof of the Poincar\'e inequality $P(1/2)$ for $\alpha$.  

What we want to do now is to refine this argument and for this matter we follow a similar argument to the one exposed in \cite{IP3}.  The main ingredient in this is to write the operator $\mathcal{N}-I$ as $(\partial_{c})^* \partial_{c}$, where $\partial_{c}$ is going to be defined below.  Before we do that we recall the definition of the non-commutative derivative introduced by Voiculescu.  It is defined on polynomials in $z$ and $z^{-1}$ over $\T$ by the rule
\[
 \partial (P Q)=(\partial P)(1\otimes Q)+(P\otimes 1)(\partial Q).
\]
Alternatively, this can be defined on monomials $z^n$ for integer values of $n$ by 
\[
\partial(z^n)=\sum_{i=0}^{n-1}z^{i}\otimes z^{n-1-i}
\]
and then extended by linearity to all polynomials.  We need to adjust this a little bit for our purpose by setting the following
\[
\partial_{c}(z^{n})=
\begin{cases} 
0 & n=0 \\
\sum_{i=1}^{n-1} z^{i}\otimes z^{n-i} & n\ge1 \\
\sum_{i=1}^{n-1} z^{-i}\otimes z^{n+i} & n\le -1,
\end{cases} 
\]
and its obvious extension by linearity to all polynomials.  To some extent, this resembles the same flavor as the derivative $\partial$, with the adjustment that the derivative of $z^{n}$ on the circle (according to our definition) is not $inz^{n-1}$ but $inz^{n}$.   If we insist, we can relate $\partial_{c}$ to $\partial$ as follows.  First, set $\mathcal{P}^{+}$ to be the set of polynomials in $z$ with no free coefficient and similarly set $\mathcal{P}^{-}$ to be the set of polynomials in $z^{-1}$ with no free coefficient.   Then, 
\[
\partial_{c}(f)=\begin{cases} \partial(z^{\pm 1}f)-f\otimes 1-1\otimes f & f\in \mathcal{P}^{\pm} \\
0 & f\equiv \text{const}.  
\end{cases}
\]

Now, let's set 
\[
 \mathcal{M}=\mathcal{N}-I.
\]
This is a non-negative operator on $L^2_0(\alpha)$.  Notice here that \eqref{ep:9} is precisely the following statement
\[
 \langle \mathcal{N}\phi ,\psi\rangle=\langle \partial \phi,\partial \psi \rangle_{\alpha\otimes \alpha}.
\]
Indeed, this is so because 
\[
 \partial(z^n)=\sum_{k=0}^{n-1}z^k\otimes z^{n-1-k}.
\]
Now, we want a similar representation for $\mathcal{M}$ instead of $\mathcal{N}$ and this is the reason for introducing $\partial_{c}$ for which a simple calculation reveals that for any polynomials $\phi,\psi\in L^{2}_{0}(\alpha)$ in $z$ and $z^{-1}$, 
\begin{equation}\label{e:pe:8}
 \langle \mathcal{M}\phi,\psi \rangle=\langle \partial_c \phi,\partial_c \psi \rangle_{\alpha\otimes \alpha}.
\end{equation}

Now, observe that for any polynomials $\phi,\psi$ in $z$ and $z^{-1}$, we have $\phi',\psi'\in L^{2}_{0}(\alpha)$ and then, 
\[
 \langle \mathcal{E}\phi',\psi' \rangle=\langle (I+\mathcal{M})^{-1}\phi',\psi' \rangle=\langle \phi',\psi' \rangle-\langle \mathcal{M}(I+\mathcal{M})^{-1}\phi',\psi' \rangle.
\]
Furthermore, the key now is to write 
\begin{equation}\label{e:pe:1}
\langle \mathcal{M}(I+\mathcal{M})^{-1}\phi',\psi' \rangle=\langle \partial_c(I+\mathcal{M})^{-1}\phi',\partial_c\psi' \rangle .
\end{equation}

The next step is to commute the operators $\partial_c$ with $(I+\mathcal{M})^{-1}$.  To do so, we need to commute $\mathcal{M}$ with $\partial_c$.  The first obstacle here is that $\mathcal{M}$ sends functions into functions,  while $\partial_c$ sends functions into tensors.  Thus we need to extend the operator $\mathcal{M}$ to tensors.  We do this by the following recipe: 
\[
 \mathcal{M}^{(2)}(P\otimes Q)=(\mathcal{M}P)\otimes Q+P\otimes (\mathcal{M}Q).  
\]
With this definition we claim that 
\begin{equation}\label{e:ep:2}
 \partial_c\mathcal{M}=(\mathcal{M}^{(2)}+I)\partial_c.
\end{equation}
Indeed, it is sufficient to check this on monomials $\phi(z)=z^n$. For $n=-1,0,1$, both sides are equal to $0$.  For $n\ge2$, $\mathcal{M}\phi=(n-1)\phi$, while $\partial_c\phi=\sum_{i=1}^{n-1}z^i\otimes z^{n-i}$ and $\mathcal{M}^{(2)}(z^i\otimes z^{n-i})=(n-2)z^i\otimes z^{n-i}$ which proves the claim.  In the case $n\le -2$, $\mathcal{M}\phi=(-n-1)\phi$ while $\partial_c\phi=\sum_{i=1}^{-n+1}z^{-i}\otimes z^{n+i}$ and $\mathcal{M}^{(2)}(z^{-i}\otimes z^{n+i})=(-n-2)z^i\otimes z^{n+i}$ which again confirms the identity. 

With these considerations, we now see that 
\[
\partial_c(\mathcal{M}+I)=(\mathcal{M}^{(2)}+2I)\partial_c,
\]
and thus 
\[
 (\mathcal{M}^{(2)}+2I)^{-1}\partial_c=\partial_c(\mathcal{M}+I)^{-1}.
\]
Therefore, if go back to \eqref{e:pe:1}, we continue
\[
 \langle \mathcal{M}(\mathcal{M}+I)^{-1}\phi',\psi' \rangle=\langle (\mathcal{M}^{(2)}+2I)^{-1}\partial_c\phi',\partial_c\psi' \rangle.
\]

Now we repeat this by writing 
\[
 \langle (\mathcal{M}^{(2)}+2I)^{-1}\partial_c\phi',\partial_c\psi' \rangle =\frac{1}{2}\langle \partial_c\phi',\partial_c\psi' \rangle-\frac{1}{2}\langle \mathcal{M}^{(2)}(\mathcal{M}^{(2)}+2I)^{-1}\partial_c\phi',\partial_c\psi' \rangle.
\]

Next, we want to extend this to higher orders.  To do this, we define inductively the operators $\mathcal{M}^{(k)}$ by the rule
\[
 \mathcal{M}^{(k)}(P\otimes Q)=(\mathcal{M}^{(k-1)}P)\otimes Q+P\otimes (\mathcal{M}Q)
\]
for any $(k-1)$ tensor $P$ and any $1$ tensor $Q$.   In addition to these, we define higher versions of $\partial_{c}$  by 
\[
 \partial_c^{(k)}=(\partial_c\otimes I^{(k)})\partial_c^{(k-1)}.
\]

The main commutation relation is that
\begin{equation}\label{e:ep:3}
 (\partial_c \otimes I^{(k-1)})\mathcal{M}^{(k)}=(\mathcal{M}^{(k+1)}+I)(\partial_c
\otimes I^{(k-1)}).
\end{equation}

To see why this is so, we use induction.  For $k=1$, this is exactly \eqref{e:ep:2}.   If we assume the statement to be true for $k\ge1$, then, by the definition of $\mathcal{M}^{(k)}$, we can write 
\[
\begin{split}
 (\partial_c \otimes I^{(k-1)})\mathcal{M}^{(k)}&=\partial_c\otimes I^{(k-1)} (\mathcal{M}^{(k-1)}\otimes I+I^{(k-1)}\otimes \mathcal{M})=((\partial_c\otimes I^{(k-2)})\mathcal{M}^{(k-1)})\otimes I +\partial_c\otimes I^{(k-2)}\otimes \mathcal{M} \\ 
 &=((\mathcal{M}^{(k)}+I)(\partial_c\otimes I^{(k-1)}) +\partial_c\otimes I^{(k-2)}\otimes \mathcal{M} \\ 
 & = (\mathcal{M}^{(k)}\otimes I+ I^{(k)}\otimes\mathcal{M})\partial_c\otimes I^{(k-1)}  +\partial_c\otimes I^{(k-1)} = (\mathcal{M}^{(k+1)}+I) \partial_c\otimes I^{(k-1)}
 \end{split}
\]
which is the proof of \eqref{e:ep:3}.  

Notice now that from  \eqref{e:ep:3} we get 
\begin{equation}\label{e:ep:6}
 (\partial_c \otimes I^{(k-1)})(\mathcal{M}^{(k)}+kI)=(\mathcal{M}^{(k+1)}+(k+1)I)(\partial_c
\otimes I^{(k-1)}).
\end{equation}

The next fact is that 
\begin{equation}\label{e:ep:4}
 \langle \mathcal{M}^{(k)}\partial_c^{(k-1)}\phi,\partial_c^{(k-1)} \psi\rangle = k\langle \partial_c^{(k)}\phi,\partial_c^{(k)}\psi \rangle.
\end{equation}

This can be proved starting with \eqref{e:ep:3} written in the form
\begin{equation}\label{5.7bis}
 \mathcal{M}^{(k)}(\partial_c\otimes I^{(k-2)})=(\partial_c\otimes I^{(k-2)})(\mathcal{M}^{(k-1)}-I),
\end{equation}
which multiplied on the right by $\partial_c^{(k-1)}$ and used iteratively yields 
\[
 \mathcal{M}^{(k)}\partial_c^{(k-1)}=\partial_c^{(k-1)}(\mathcal{M}-k+1).
\]
Thus, checking \eqref{e:ep:4} becomes now equivalent to 
\[
 \langle \partial_c^{(k-1)}(\mathcal{M}-k+1)\phi,\partial_c^{(k-1)}\psi \rangle = k\langle \partial_c^{(k)}\phi,\partial_c^{(k)}\psi \rangle.
\]
It is now sufficient to do this for the case $\phi(z)=\phi_n(z)=z^n$ and $\psi(z)=\psi_m(z)=z^m$ with $m,n$ integer numbers.  Since $\mathcal{M}\phi_n=(|n|-1)\phi_n$, in the case of $m$ and $n$ have different signs, both sides above are 0.   Thus, it suffices to consider the case when both are positive or both are negative.  Since the case of negative values is similar, we treat only the case of $m,n$ both positive.  In this case, the equality to be proven becomes   
\[
(n-k)\langle \partial_c^{(k-1)}\phi_n,\partial_c^{(k-1)}\psi_n \rangle = k\langle \partial_c^{(k)}\phi_n,\partial_c^{(k)}\psi_n \rangle.
\]
It is now a straightforward exercise to show by induction that 
\[
 \partial_c^{(k)}\phi_n=\sum_{\substack{i_1+i_2+\dots +i_{k+1}=n\\ i_{1},i_{2}\dots,i_{k+1}\ge1}}z^{i_1}\otimes z^{i_2}\otimes \cdots z^{i_{k+1}}.
\]
Thus, the equality to be checked becomes
\begin{equation}\label{e:ep:5}
 (n-k)N_{k-1,n-1}\delta_{n,m}=N_{k,n-1}\delta_{n,m}
\end{equation}
where $N_{k,l}$ is the number of writings of $l=a_1+a_2+\dots +a_{k+1}$ with $a_{1},a_{2},\dots,a_{k+1}\ge1$.   It is very easy to verify that $N_{k,l}={l-1 \choose k}$ as one can see from equating the coefficients of $t^{l}$ from the equality
\[
\underbrace{(t+t^{2}+t^{3}+\dots)(t+t^{2}+t^{3}+\dots)\dots (t+t^{2}+t^{3}+\dots)}_{k+1\text{ times}}=\frac{t^{k+1}}{(1-t)^{k+1}}=\sum_{u\ge k}{u \choose k }t^{u+1}.
\]
With this, the equality \eqref{e:ep:5} is completed and in turn \eqref{e:ep:4} is checked. 

After these preliminaries, the main result of this section is contained in the following.  

\begin{theorem}  For any $k\ge1$ and polynomials $\phi,\psi$, we have 
\begin{equation}
\begin{split}
\langle \mathcal{N}\phi,\psi \rangle=&\langle \phi',\psi' \rangle - \frac{1}{2}\langle \partial_{c}\phi',\partial_{c}\psi'\rangle_{\alpha^{\otimes 2}}+\dots +  \frac{(-1)^{k}}{k}\langle \partial_{c}^{(k-1)}\phi',\partial_{c}^{(k-1)}\psi'\rangle_{\alpha^{\otimes (k-1)}} \\ 
&+\frac{(-1)^{k}}{k}\langle \mathcal{M}^{(k)}(\mathcal{M}^{(k)}+kI)^{-1}\partial_{c}^{(k-1)}\phi',\partial_{c}^{(k-1)}\psi'\rangle_{\alpha^{\otimes (k-1)}}. 
\end{split}
\end{equation}
In particular for any $k\ge1$ the following holds
\[
\sum_{l=1}^{2k}\frac{(-1)^{l-1}}{l}\|\partial^{(l-1)}_{c}\phi^\prime\|^{2}_{\alpha^{\otimes l}}
\le \langle\mathcal{N}\phi,\phi 
\rangle\le 
\sum_{l=1}^{2k-1}\frac{(-1)^{l-1}}{l}\|\partial^{(l-1)}_{c}\phi^\prime\|^{2}_{\alpha^{\otimes l}}.
\]
This is a refinement of the free Poincar\'e inequality for the Haar measure $\alpha$.  
\end{theorem} 

\begin{proof}
The proof is an induction on $k$, use of the key facts \eqref{e:ep:4} and \eqref{e:ep:6} combined with the following  sequence of equalities 
\[
\begin{split}
\langle \mathcal{M}^{(k)}&(\mathcal{M}^{(k)}+kI)^{-1}\partial_{c}^{(k-1)}
\psi,\partial_{c}^{(k-1)}\psi \rangle_{\alpha^{\otimes k}}
\underset{\eqref{5.7bis}}{=}\langle \mathcal{M}^{(k)}\partial_{c}^{(k-1)}(\mathcal{M}+I)^{-1} \psi,\partial_{c}^{(k-1)}\psi 
\rangle_{\alpha^{\otimes k}} \\ 
&\underset{\eqref{e:ep:4}}{=}k\langle \partial_{c}^{(k)}(\mathcal{M}+I)^{-1} \psi,\partial_{c}^{(k)}\psi \rangle_{\alpha^{\otimes k}} \\ 
&\underset{\eqref{5.7bis}}{=}k\langle(\mathcal{M}^{(k+1)}+(k+1)I)^{-1}\partial_{c}^{(k)} \psi,\partial_{c}^{(k)}\psi
\rangle_{\alpha^{\otimes k}} \\
&\underset{\eqref{e:ep:6}}{=}\frac{k}{k+1}\langle \partial_{c}^{(k)} \psi,\partial_{c}^{(k)}\psi
\rangle_{\alpha^{\otimes k}} -\frac{k}{k+1}\langle(\mathcal{M}^{(k+1)}(\mathcal{M}^{(k+1)}+(k+1)I)^{-1}\partial_{c}^{(k)} \psi,\partial_{c}^{(k)}\psi\rangle_{\alpha^{\otimes k}}.
\end{split}
\]
Notice that the operator $(\mathcal{M}^{(k)}+kI)$ is not invertible on the set of all tensors, but it is so on the set of tensors on the form $\partial_{c}^{(k-1)}\phi$.   \qedhere

\end{proof}

\section{Poincar\'e's inequality under convexity assumptions}\label{S:6}

The main result of this section is the following.  
\begin{theorem}\label{t:p}  Assume $Q$ is a $C^3$ potential on $\T$ such that $Q''\ge\rho -1/2$ for some $\rho>0$.  Then for any smooth function $\phi$ on $\T$, 
\begin{equation}\label{e6:t6}
2\rho\iint\left|\frac{\phi(z)-\phi(w)}{z-w} \right|^2\alpha(dz)\alpha(dw)\le \int\left|\phi'-\int \phi'\,d\mu_Q\right|^2 d\mu_Q.
\end{equation}
\end{theorem}

\begin{proof}
The proof starts with \eqref{e:pe:0} and the observation that $\mathcal{E}1=0$, thus allowing us to write
\[
\langle \mathcal{N}\phi,\phi \rangle=\langle \mathcal{E}(\phi'-\int \phi' d\mu_{Q}), \phi'-\int \phi' d\mu_{Q} \rangle.
\] 
Furthermore, since $\mathcal{E}\le I$, it follows that 
\[
\begin{split}
\langle \mathcal{E}(\phi'-\int \phi' d\mu_{Q}), \phi'-\int \phi' d\mu_{Q} \rangle &\le \langle \phi'-\int \phi' d\mu_{Q}, \phi'-\int \phi' d\mu_{Q} \rangle \\ 
& = \left\langle \frac{1}{1-\mathcal{N}Q}\left(\phi'-\int \phi' d\mu_{Q}\right), \phi'-\int \phi' d\mu_{Q} \right\rangle_{\mu_{Q}} \\
&=\int\frac{\left|\phi'-\int \phi'\,d\mu_Q\right|^2}{1-\mathcal{N}Q} d\mu_Q,
\end{split}
\]
where in the middle we used \eqref{e:d}.  

What we need to show now is that $1-\mathcal{N}Q\ge2\rho$.  To attain this we use Corollary~\ref{c:2}, again in conjunction with \eqref{e:d}.  Using $\beta=\rho-\frac{1}{2}+\frac{1}{2\log2}$ in Corollary~\ref{c:2} and the fact that $\rho>0$, we obtain that 
\[
1-\mathcal{N}Q\ge 2\log 2\left(\rho-\frac{1}{2}+\frac{1}{2\log 2}\right). 
\]
On the other hand, since $\int Q''\,d\alpha=0$ and $Q''\ge\rho-1/2$, we learn that $\rho\le 1/2$.   This justifies that 
\[
2\log 2\left(\rho-\frac{1}{2}+\frac{1}{2\log 2}\right)-2\rho=2(\log 2-1)\left(\rho-\frac{1}{2}\right)\ge 0
\]
which then proves that $1-\mathcal{N}Q\ge2\rho$ and this in turn ends the proof.  \qedhere
\end{proof}

We should point that the constant $2\rho$ in \eqref{e6:t6} can be replaced by the larger quantity $2\log 2\left(\rho-\frac{1}{2}+\frac{1}{2\log 2}\right)$. However, in the case of $Q\equiv 0$, which is to say $\rho=1/2$, both quantities give the same values, namely $1$ which is also sharp.

\section{The modified Wasserstein Distance} \label{S:7} 

In \cite{HPU1}, the free transportation inequality on the circle is deduced as the limiting of the transportation inequalities on orthogonal groups.  To recall the main idea in there, let's denote by $U(n)$ the group of $n\times n$  unitary matrices and for any potential $Q$ on $\T$ and consider
\[
\p_{n}^{Q}(dU)=\frac{1}{Z_{n}(Q)}e^{-n\Tr Q(U)}dU. 
\]
This is going to be our reference probability measure on $U(n)$ associated to $Q$.  The classical transportation inequality applied to this reference measure should produce the free transportation inequality.   However, to apply the classical inequality one has to take into account that the classical case requires extra conditions, as for example the Bakry-Emery condition.  Even in the simplest case, $Q\equiv 0$, the reference measure does not satisfy Bakry-Emery condition because the Ricci curvature of $U(n)$ is degenerate.   Nevertheless, this is saved by the fact that $SU(n)$, the subgroup of unitary matrices of determinant $1$, has indeed constant Ricci curvature.  It is interesting to point out that the particular structure of $SU(n)$ is actually reflected into the corresponding free Log-Sobolev inequality, while the form of the transportation inequality is in fact the same as in the case of the real line (with the appropriate quantities well defined).  Particularly noticeable  here is that the free transportation inequality still uses the standard Wasserstein distance on the circle.   

What we propose in this note is a distance which seems to be, morally, a reflection of the $SU(n)$ structure versus $U(n)$ structure.  The main difference between $U(n)$ and $SU(n)$ is that the latter is a version of the former with an extra constraint.  If we look at the measures on $U(n)$ and lift them (using the exponential map) to the Lie algebra of $U(n)$, then, lifts of measures which are supported on $SU(n)$ have to satisfy a certain constraint.  At the heuristic level this is that the integral of the lifted measure must have $0$ expectation.  This last statement comes from the basic structure of $SU(n)$, namely matrices of determinant $1$ and matrices in the Lie algebra $su(n)$ have $0$ trace.     

Furthermore, the passage from the transportation inequality on $U(n)$ to the free transportation on the circle, involves looking at the distribution of the eigenvalues of the random matrix.  According to the above paragraph  (at the heuristic level) the lift to the Lie algebra level implies that the lifted measure of the eigenvalues satisfies the extra constraint that the mean of this measure is $0$.  This is going to be our guiding leitmotiv in the following alteration of the definition of the Wasserstein distance. 

On a metric space $(X,d)$ the standard Wasserstein distance $W_{p}$ for $p\ge1$ is defined as 
\[
W_{p}(\mu,\nu)=\inf_{\pi\in\Pi(\mu,\nu)}\left( \iint d(x,y)^{p}\pi(dx,dy) \right) ^{1/p}
\]
where the set $\Pi(\mu,\nu)$ is the set of probability measures on $X\times X$ with marginals $\mu$ and $\nu$.  We will use this in the case of $X=\R$ with the length distance and also in the case of $X=\T$ and the arc-length distance.   We will do this tacitly and we will mention this explicitly where confusion arises.  

We will modify this definition in what follows and for this purpose we proceed first with more notations.  

In the sequel, $\exp:\R\to\T$ denotes the map $\exp(x)=e^{ix}$.   For a given $u$, we set $\exp_{u}$ the restriction of $\exp$ to $[u,u+2\pi)$.  Furthermore, for a measure $\mu$ on $\T$, we define the lift measure $\bar{\mu}$ on $\R$ by 
\[
\bar{\mu}(A)=\sum_{n\in\Z}\mu(\exp(A\cap [2n\pi,2(n+1)\pi)).
\]

For a measure $\mu$ on $\T$, the lift measure $\bar{\mu}_{u}$ is simply the restriction of $\bar{\mu}$ to the interval $[u,u+2\pi)$.

\begin{definition} For any two probability measures $\mu,\nu$ on $\T$, we define for any $p\ge1$
\begin{equation}\label{e:W}
\mathcal{W}_{p}(\mu,\nu)=\sup\left\{ W_{p}(\bar{\mu}_{u},\bar{\nu}_{v}): u,v\in[0,2\pi] \text{ such that } \int x\,\bar{\mu}_{u}(dx)=\int x\,\bar{\nu}_{v}(dx)\right\},
\end{equation}
where $W_{p}(\bar{\mu}_{u},\bar{\nu}_{v})$ is the standard Wasserstein distance on the real line.   We use here the convention that the supremum over the empty set is $+\infty$, thus $\mathcal{W}_{p}(\mu,\nu)$ can be $+\infty$.  In fact, $\mathcal{W}_{p}(\mu,\nu)=+\infty$ if and only if for any $u,v$, $\int x\,\bar{\mu}_{u}(dx)\ne\int x\,\bar{\nu}_{v}(dx)$.

We denote by $\mathcal{P}_{as}(\T)$ the set of atomless probability measures on $\T$. 
\end{definition}

Below we collect a set of properties of this new object.  

\begin{proposition}\label{p:W}
\begin{enumerate} 

\item $\mathcal{W}_{p}(\mu,\nu)=\mathcal{W}_{p}(\nu,\mu)$.

\item  For any probability measures $\mu,\nu$ on $\T$, 
\begin{equation}\label{e:t1}
W_{p}(\mu,\nu)\le \mathcal{W}_{p}(\mu,\nu).
\end{equation}

\item $\mathcal{W}_{p}(\mu,\nu)=0$ if and only if $\mu=\nu$.

\item If $\mu=\delta_{a}$ and $\nu=\delta_{b}$, then $\mathcal{W}_{p}(\delta_{a},\delta_{b})=\infty$ unless $a=b$.

\item For any $a\in \T$, $\mathcal{W}_{p}(\delta_{a},\alpha)=\left(\frac{2\pi^{p+1}}{p+1}\right)^{1/p}$ for any $p\ge1$.  In particular, for $a,b\in\T$ with $a\ne b$ we have $\mathcal{W}_{p}(\delta_{a},\delta_{b})>\mathcal{W}_{p}(\delta_{a},\alpha)+\mathcal{W}_{p}(\alpha,\delta_{b})$.  This means that $\mathcal{W}_{p}$ is not a distance on the set of probability measures. 

\item If $\mu,\nu\in \mathcal{P}_{as}(\T)$, then there exists $t\in[0,2\pi]$ such that $\int x\,\bar{\mu}_{t}(dx)=\int x\,\bar{\nu}_{t}(dx)$.  Moreover, for any $u\in[0,2\pi]$, there exists $v\in[0,2\pi]$ such that $\int x\,\bar{\mu}_{u}(dx)=\int x\,\bar{\nu}_{v}(dx)$.  

If in addition, $u,v\in[0,2\pi]$ and $u',v'\in[0,2\pi]$ are two pairs such that $\int x\,\bar{\mu}_{u}(dx)=\int x\,\bar{\nu}_{v}(dx)$ and $\int x\,\bar{\mu}_{u'}(dx)=\int x\,\bar{\nu}_{v'}(dx)$, then $W_{p}(\bar{\mu}_{u},\bar{\nu}_{v})=W_{p}(\bar{\mu}_{u'},\bar{\nu}_{v'})$.  In other words, for any $u,v\in[0,2\pi]$ such that $\int x\,\bar{\mu}_{u}(dx)=\int x\,\bar{\nu}_{v}(dx)$ and any $p\ge1$
\[
\mathcal{W}_{p}(\mu,\nu)=W_{p}(\bar{\mu}_{u},\bar{\nu}_{v}).  
\]
In particular, $\mathcal{W}_{p}(\mu,\nu)<\infty$ and the supremum in the definition of $\mathcal{W}_{p}(\mu,\nu)$ is attained.  

\item If $\mu,\zeta,\nu$ are three probability measures in $\mathcal{P}_{as}(\T)$, then  
\[
\mathcal{W}_{p}(\mu,\nu)\le \mathcal{W}_{p}(\mu,\zeta)+\mathcal{W}_{p}(\zeta,\nu).  
\]

\item On $\mathcal{P}_{as}(\T)$, the topology induced by $\mathcal{W}_{p}$ is the topology of weak convergence.   More precisely, if $\{\mu_{n}\}_{n\ge1}\subset\mathcal{P}_{as}(\T)$ and $\mu\in \mathcal{P}_{as}(\T)$, then $\mu_{n}$ converges weakly to $\mu$ if and only if $\mathcal{W}_{p}(\mu_{n},\mu)$ converges to $0$.   

\item For any measure $\mu\in \mathcal{P}_{as}(\T)$, $\mathcal{W}_{p}(\mu,\alpha)=W_{p}(\mu,\alpha)$.  
\end{enumerate}

\end{proposition}

\begin{proof}

\begin{enumerate}
\item It is clear from the definition.  
\item The point is that if we denote by $d(z,w)$ the distance on $\T$, then 
\[
d(z,w)=\inf_{k,l\in\Z} |x+2\pi k -(y+2\pi l) |
\]
for any $x,y$ such that $z=e^{ix}$ and $y=e^{iy}$. In particular, 
\begin{equation}\label{e:t2c}
d(z,w)\le |x-y|
\end{equation}
for any $x,y$ such that $z=e^{ix}$ and $w=e^{iy}$. 

If $\mathcal{W}_{p}(\mu,\nu)=+\infty$, there is nothing to prove.  In the case $\mathcal{W}_{p}(\mu,\nu)<+\infty$ we can pick any $u,v$ such that $\int x\bar{\mu}_{u}(dx)=\int x\bar{\nu}_{v}(dx)$, and $\bar{\gamma}$ any plan on $\R\times\R$ with marginals $\bar{\mu}_{u}$ and $\bar{\nu}_{v}$.  Then $\gamma=(\exp_{u}\times \exp_{v})_{\#}\bar{\gamma}$ is a plan on $\T\times \T$ with marginals $\mu$ and $\nu$.  Now,  \eqref{e:t2c} shows that for any $u,v\in[0,2\pi]$ such that $\int x\bar{\mu}_{u}(dx)=\int x\bar{\nu}_{v}(dx)$, we get that 
\[
W_{2}(\mu,\nu)\le W_{2}(\bar{\mu}_{u},\bar{\nu}_{v})
\]
and consequently, 
\[
W_{2}(\mu,\nu)\le \mathcal{W}_{2}(\mu,\nu)
\]
for any $\mu,\nu$.  

\item This is clear from the fact that $W_{p}(\mu,\nu)=0$ if and only if $\mu=\nu$.  This fact combined with \eqref{e:t1} yields the claim.  

\item For any $u\in[0,2\pi]$, $\int x (\bar{\delta_{a}})_{u}(dx)=\bar{a}$ with $\exp_{u}(\bar{a})=a$.  Similarly, for any $v\in [0,2\pi]$, we also have $\int x (\bar{\delta_{b}})_{v}(dx)=\bar{b}$ with $\exp_{v}(\bar{b})=b$.  Therefore, we need to have $\bar{a}=\bar{b}$ which means that $a=b$ and obviously, $\mathcal{W}_{p}(\mu,\nu)=0$.  If $a\ne b$, we can not find $u,v$ such that the integral condition is satisfied, thus $\mathcal{W}_{p}(\delta_{a},\delta_{b})=+\infty$.   

\item In the first place, for any $v$, $\bar{\alpha}_{v}$ is the uniform probability measure on $[v,v+2\pi)$.   We need to find $u$ and $v$ such that $\int x (\bar{\delta_{a}})_{u}(dx)=\bar{a}=\int x \bar{\alpha}_{v}(dx)=v+\pi$ from which $v=(\bar{a}-\pi)$ and then $W_{p}^{p}((\bar{\delta_{a}})_{u},\bar{\alpha}_{v})=\int_{v}^{v+2\pi}|x-\bar{a}|^{p}dx=\frac{2\pi^{p+1}}{p+1}$.  The rest is trivial.  

\item The key here is the following equation 
\begin{equation}\label{e:t4}
\int x\bar{\mu}_{u}(dx)=2\pi\bar{\mu}_{0}([0,u))+\int x\bar{\mu}_{0}(dx),
\end{equation}
which is a consequence of the fact that $\bar{\mu}_{u}(A)=\bar{\mu}_{0}(A\cap [u,2\pi))+\bar{\mu}_{0}((A-2\pi)\cap[0,u])$ for any $A\subset[u,u+2\pi)$.   Now, if $\mu,\nu$ do not have atoms, then we want to find $t\in[0,2\pi]$ such that 
\begin{equation}\label{e:t5bb}
\bar{\mu}_{0}([0,t))-\bar{\nu}_{0}([0,t))=-\frac{1}{2\pi}\left(\int x\bar{\mu}_{0}(dx)-\int x\bar{\nu}_{0}(dx) \right)
\end{equation}
Since $\mu$ and $\nu$ do not have atoms, the function 
\[
f(t)=\bar{\mu}_{0}([0,t))-\bar{\nu}_{0}([0,t))
\]
is a continuous function of $t$.  In addition, a simple argument shows that
\[
\frac{1}{2\pi}\int_{[0,2\pi)}f(t)dt=-\frac{1}{2\pi}\left(\int x\bar{\mu}_{0}(dx)-\int x\bar{\nu}_{0}(dx) \right).
\]
Therefore, continuity of $f$ guarantees that there is an intermediate point $t\in[0,2\pi]$ such that \eqref{e:t5bb} is satisfied.  

Now, using \eqref{e:t4}, we get that $\int x\,\bar{\mu}_{u}(dx)=\int x\,\bar{\nu}_{v}(dx)$ if and only if 
\begin{equation}\label{e:t6}
\bar{\mu}_{0}([0,u))-\bar{\nu}_{0}([0,v))=-\frac{1}{2\pi}\left(\int x\,\bar{\mu}_{0}(dx)-\int x\,\bar{\nu}_{0}(dx) \right).
\end{equation}
Given the solution $t$ to \eqref{e:t5bb}, the above equation becomes equivalent to 
\[
\bar{\mu}_{0}([t,u))=\bar{\nu}_{0}([t,v))
\] 
if $u>t$ and 
\[
\bar{\mu}_{0}([u,t))=\bar{\nu}_{0}([v,t))
\]
if $u\le t$.  Again, since the measures do not have atoms, we deduce that given $u$, there is at least one $v$ such that \eqref{e:t6} is satisfied.  

Assume now that $u,v\in [0,2\pi]$ are two values such that $\int x\,\bar{\mu}_{u}(dx)=\int x\,\bar{\nu}_{v}(dx)$.  There is a non-decreasing transportation map denoted by $\theta_{u,v}$ which maps $\bar{\mu}_{v}$ into $\bar{\nu}_{v}$. This means, according to \cite[page 75, equation (2.49)]{Villani2}, that 
\begin{equation}\label{e:t7}
\bar{\mu}_{u}([u,x))=\bar{\nu}_{v}([v,\theta_{u,v}(x)))
\end{equation}
for any $x\in[u,u+2\pi)$ and 
\[
W_{p}^{p}(\bar{\mu}_{u},\bar{\nu}_{v})=\int (x-\theta_{u,v}(x))^{p}\,\bar{\mu}_{u}(dx).
\]
It is clear that $\theta_{u,v}:[u,u+2\pi)\to[v,v+2\pi)$ because $\bar{\mu}_{u}$ is supported on $[u,u+2\pi)$ and $\bar{\nu}_{v}$ is supported on $[v,v+2\pi)$.   

Next assume that $u,v$ and $u',v'$ are two pairs such that $\int x\,\bar{\mu}_{u}(dx)=\int x\,\bar{\nu}_{v}(dx)$ and $\int x\,\bar{\mu}_{u'}(dx)=\int x\,\bar{\nu}_{v'}(dx)$.  We may assume for example that $u\le u'$ and then according to \eqref{e:t6}, we obtain that 
\[
\bar{\mu}_{0}([u,u'))=\bar{\nu}_{0}([v,v')).
\]
If $\bar{\mu}_{u}([u,u'))=0$, then we also have that $\bar{\nu}_{0}([v\wedge v',v\vee v'))=0$ and in this case, if $v'<v$ it is easy to see that $\theta_{u,v}$ also maps $\bar{\mu}_{u'}$ into $\bar{\nu}_{v'}$ and thus $W_{2}(\bar{\mu}_{u},\bar{\nu}_{v})=W_{2}(\bar{\mu}_{u'},\bar{\nu}_{v'})$.   Therefore, without loss of generality we may simply assume that $v\le v'$.    By the definition of $\bar{\mu}_{u}$, we also obtain that 
\begin{equation}\label{e:t10}
\bar{\mu}_{u}([u,u'))=\bar{\nu}_{v}([v,v')).  
\end{equation}

Let $F_{u}$ be the cumulative function of $\bar{\mu}_{u}$ and $G_{v}$ the cumulative function of $\bar{\nu}_{v}$.  Similarly denote by $F_{u'}$ and $G_{v'}$ the cumulative functions of $\bar{\mu}_{u'}$ and $\bar{\nu}_{v'}$.   Now, \eqref{e:t10} reads as
\begin{equation}\label{e:t13}
F_{u}(u')=G_{v}(v')
\end{equation}
Also, from \cite[Section 2.2]{Villani2} and with the notations from there, we know that 
\begin{equation}\label{e:t12}
\theta_{u,v}=G_{v}^{-1}\circ F_{u}
\end{equation} 
where $G^{-1}$ is the generalized inverse of $G$.  

The main point in the proof now is the following relation
\begin{equation}\label{e:t11}
F_{u'}(x)=\begin{cases}
F_{u}(x)-F_{u}(u') & u'\le x< u+2\pi \\
F_{u}(x-2\pi)+1-F_{u}(u') & u+2\pi \le x <u'+2\pi.
\end{cases}
\end{equation}
To see why this is true, we just need to observe that $F_{u'}(x)=\bar{\mu}_{u'}([u',x))=\bar{\mu}_{u}([u,x)) - \bar{\mu}_{u}([u,u'))$ for $u'\le x<u+2\pi$ and $\bar{\mu}_{u'}([u',x))=\bar{\mu}([u',u+2\pi))+\bar{\mu}([u+2\pi,x))=1-F_{u}(u')+\bar{\mu}([u,x-2\pi))$.   Using the similar writing for $G_{v}$ and $G_{v'}$ it is now a simple exercise to check that 
\begin{equation}\label{e:t11b}
G_{v'}^{-1}(y)=\begin{cases}
G_{v}^{-1}(y+G_{v}(v')) & 0\le y< G_{v}(v') \\
G_{v}^{-1}(y-1+G_{v}(v'))+2\pi & G_{v}(v') \le y <1.
\end{cases}
\end{equation}

Combining \eqref{e:t13}, \eqref{e:t12}, \eqref{e:t11} and \eqref{e:t11b} we obtain 
\begin{equation}\label{e:t8}
\theta_{u',v'}(x)=\begin{cases}
\theta_{u,v}(x) & u'\le x< u+2\pi \\
\theta_{u,v}(x-2\pi)+2\pi & u+2\pi \le x<u'+2\pi.
\end{cases}
\end{equation}
Once this is obtained, a verification reveals that
\[
W_{p}^{p}(\bar{\mu}_{u},\bar{\nu}_{v})=\int_{[u,u+2\pi)}(x-\theta_{u,v}(x))^{p}\bar{\mu}(dx)=\int_{[u',u'+2\pi)}(x-\theta_{u',v'}(x))^{p}\bar{\mu}(dx)=W_{p}^{p}(\bar{\mu}_{u'},\bar{\nu}_{v'}),
\]
which is what we wanted to show.  The rest is obvious.  

\item Take $u,v,t$ such that $\int x\,\bar{\mu}_{u}(dx)=\int x\,\bar{\zeta}_{t}(dx)$ and for this fixed $t$ choose $v$ such that $\int x\,\bar{\zeta}_{t}(dx)=\int x\,\bar{\nu}_{v}(dx)$ which is always possible as we pointed out above.  Then, the rest follows from the fact that $\int x\,\bar{\mu}_{u}(dx)=\int x\,\bar{\nu}_{v}(dx)$ and according to the previous item, 
\[
\mathcal{W}_{p}(\mu,\nu)=W_{p}(\bar{\mu}_{u},\bar{\nu}_{v})\le W_{p}(\bar{\mu}_{u},\bar{\zeta}_{t})+W_{p}(\bar{\zeta}_{t},\bar{\nu}_{v})=\mathcal{W}_{p}(\mu,\zeta)+\mathcal{W}_{p}(\zeta,\nu).  
\]

\item Assume that $\mu_{n}$ is a sequence of probability measures in $\mathcal{P}_{as}(\T)$ converging weakly to $\mu\in \mathcal{P}_{as}(\T)$ and similarly  $\nu_{n}$ converging weakly to $\nu$ with, $\nu_{n},\nu\in\mathcal{P}_{as}(\T)$.  We claim that 
\[
\mathcal{W}_{p}(\mu_{n},\nu_{n})\xrightarrow[n\to\infty]{} \mathcal{W}_{p}(\mu,\nu).
\]

To see this, we argue by contradiction.  Assume that for some $\epsilon>0$ and a subsequence $n_{k}$ we have  
\[
|\mathcal{W}_{p}(\mu_{n_{k}},\nu_{n_{k}}) - \mathcal{W}_{p}(\mu,\nu)|\ge\epsilon.
\] 

Now, we can find $u_{n_{k}}\in[0,2\pi)$ such that 
\[
\int x\,(\bar{\mu}_{n_{k}})_{u_{n_{k}}}(dx)=\int x\,(\bar{\nu}_{n_{k}})_{u_{n_{k}}}(dx).
\] 
On a further subsequence of $n_{k}$, say $n_{k,1}$, we can assume that $u_{n_{k,1}}$ converges to $u$.  By the fact that the measures involved do not have atoms, it is an easy exercise to check that $(\bar{\mu}_{n_{k,1}})_{u_{n_{k,1}}}$ converges weakly to $\bar{\mu}_{u}$ and similarly, $(\bar{\nu}_{n_{k,1}})_{u_{n_{k,1}}}$ converges to $\bar{\nu}_{u}$.   In addition, again by the fact that the measures involved do not have atoms, 
\[
\int x \,\bar{\mu}_{u}(dx)=\int x\,\bar{\nu}_{u}(dx).
\]
Now, since 
\[
\mathcal{W}_{p}(\mu_{n_{k}},\nu_{n_{k}})=W_{p}((\bar{\mu}_{n_{k}})_{u_{n_{k}}},(\bar{\nu}_{n_{k}})_{u_{n_{k}}}) \text{ and }\mathcal{W}_{p}(\mu,\nu)=W_{p}(\bar{\mu}_{u},\bar{\nu}_{u}),
\]
we get by taking the limits and using the fact that $W_{p}$ is continuous with respect to the weak limits that 
\[
\mathcal{W}_{p}(\mu_{n_{k,1}},\nu_{n_{k,1}})\xrightarrow[k\to\infty]{}\mathcal{W}_{p}(\mu,\nu),
\]
which is a contradiction with our initial assumption.   This gives the first implication of the claim.   

For the reverse implication, if we have 
\[
\mathcal{W}_{p}(\mu_{n},\mu)\xrightarrow[k\to\infty]{}0
\]
then $\mu_{n}$ must converge in weak topology to $\mu$.  Otherwise, we can extract a subsequence $n_{k}$ such that for some fixed $\epsilon>0$, we have
\[
W_{p}(\mu_{n_{k}},\mu)\ge\epsilon
\]
where here $W_{p}$ is the standard Wasserstein distance on the space of probability measures on $\T$.   Indeed, this is so because the distance $W_{p}$ induces the topology of weak convergence on $\T$.   Now, using a further subsequence we can find a subsequence of $\mu_{n_{k}}$ which is weakly convergent to some measure $\nu$.  Using the first implication and the convergence above we obtain in the first place that $W_{p}(\mu,\nu)\ge\epsilon$ and on the other hand $\mathcal{W}_{p}(\mu,\nu)=0$ which is a contradiction.  Thus $\mu_{n}$ must converge weakly to $\mu$.

\item Assume $u$ is such that $\int x\bar{\mu}_{u}(dx)=\int x \bar{\alpha}_{u}(dx)$.  Because $\alpha$ is the Haar measure on $\T$, we certainly have $\int x \bar{\alpha}_{u}(dx)=u+\pi$.  Now if we take the transport map $\theta$ which maps $\bar{\alpha}_{u}$ into $\bar{\mu}_{u}$, we claim that $|\theta(x)-x|\le \pi$.  Indeed, if this were not so, then we would have $z$, say for simplicity, $z\in[u,u+\pi)$ such that $\theta(z)>z+\pi$.  Since $\theta$ is non-decreasing, we also get that $\theta(y)\ge z+\pi$ for all $y\in [z,u+2\pi)$.   On the other hand, 
\[
\begin{split}
2\pi(u+\pi)&=\int_{u}^{u+2\pi}\theta(x)dx=\int_{u}^{z}\theta(x)dx+\int_{z}^{u+2\pi}\theta(x)dx \\ 
&\ge \int_{u}^{z}udx+\int_{z}^{u+2\pi}(z+\pi)dx=u(z-u)+(z+\pi)(u+2\pi-z) \\
\end{split}
\]
which leads to $z\ge u+\pi$ in contradiction to $z\in[u,u+\pi)$.  A similar argument can be run to prove that we can not have $z\ge u+\pi$ such that $|\theta(z)-z|>\pi$.   Since the map $\theta$ is such that $|\theta(x)-x|\le \pi$, it means that the segment on the circle between $x$ and $\theta(x)$ is a geodesic segment and such $\theta$ is the transportation map for $W_{2}(\mu,\nu)$ on the circle.   \qedhere
 
\end{enumerate}

\end{proof}

As in the classical case there is a dual formulation of the distance $\mathcal{W}_{p}$ which is given next.   
In the following statement periodic functions are periodic of period $2\pi$. 

\begin{theorem}\label{t:W}
\begin{enumerate}
\item Given $p\ge1$ and two measures $\mu,\nu\in\mathcal{P}(\T)$ such that $\mathcal{W}_{p}(\mu,\nu)<\infty$, we have 
\begin{equation}\label{e:t20}
\mathcal{W}_{p}^{p}(\mu,\nu)=\sup_{(\lambda,f,g)}\left\{ \int f\,d\mu+\int g\,d\nu:  f(x)+g(y)\le |x-y|^{p}+\lambda(x-y), \,\forall x,y\in\R \right\},
\end{equation}
where the supremum is taken over all triples $(\lambda,f,g)$ with $\lambda\in \R$ and  $f,g:\T\to\R$ smooth functions (identified as periodic continuous functions on $\R$).  
\item Assume that $\lambda\in\R$ and $f$ is a continuous bounded function on the real line.  Define now
\[
(\mathcal{U}^{\lambda}_{t}f)(x)=\inf_{y\in\R}\left\{f(y)+\frac{(x-y)^{2}}{t}+\lambda(x-y)\right\}.
\] 
Then, $f_{t}(x)=(\mathcal{U}^{\lambda}_{t}f)(x)$ satisfies
\[
\partial_{t}f_{t}=-\frac{(f_{t}'-\lambda)^{2}}{4} \text{ with }f_{0}=f.
\]
In addition, if $f$ is a $T$-periodic functions, then $\mathcal{U}^{\lambda}_{t}f$ is also a $T$-periodic function.  
\item We have for any measures $\mu,\nu\in \mathcal{P}(\T)$ with $\mathcal{W}_{2}(\mu,\nu)<\infty$, 
\[
\mathcal{W}_{2}^{2}(\mu,\nu)=\sup_{(\lambda,f)}\left\{ \int \mathcal{U}^{\lambda}_{1}f\,d\mu- \int f\,d\nu \right\},
\]
where the supremum is taken over all $\lambda\in\R$ and periodic smooth functions $f$.   

\item For $p=1$ the following characterization holds: 
\begin{equation}\label{e:t21}
\mathcal{W}_{1}(\mu,\nu)=\sup_{\lambda,f}\left\{ \int f\,d(\mu-\nu): |f(x)-f(y)-\lambda(x-y)|\le |x-y| \right\}
\end{equation}
where the supremum is taken over all $\lambda\in\R$ and continuous periodic $f$ and we assume that $\mathcal{W}_{1}(\mu,\nu)<\infty$.   
\end{enumerate}
\end{theorem}

Notice that if $\mu,\nu\in\mathcal{P}_{as}(\T)$, then for any $p\ge1$, we obtain that $\mathcal{W}_{p}(\mu,\nu)<\infty$ and thus  the Theorem applies.  

\begin{proof}
\begin{enumerate}
 \item Notice that it is enough to prove \eqref{e:t20} only for the case of functions $f,g$ which are Borel measurable and periodic.  This is a consequence of the fact that for any Borel bounded measurable and periodic function $f:\R\to\R$ we can find a sequence of smooth periodic functions $f_{n}:\R\to\R$ such that 
 \[
 f_{n}(x)\le f(x)\text{ and } \lim_{n\to\infty}f_{n}(x)=f(x).
 \]
 This can be seen by approximating first from below the function $f$ by periodic step functions and then each step function, again from below, by a smooth one.  As a technical point, to justify the proper convergence, we notice that  from $f(x)+g(y)\le |x-y|^{p}+\lambda(x-y)$ and periodicity,  it must be that  $f$ and $g$ are bounded. 
 
With this we can simply enlarge the class of functions in \eqref{e:t20} to any Borel measurable and periodic functions $f,g$ satisfying $f(x)+g(y)\le |x-y|^{p}+\lambda(x-y)$ for all $x,y\in\R$.  
 
 Observe first that for any $u,v$, functions $f,g:\T\to\R$ and $\lambda\in\R$,
\[
\int f\,d\mu+\int g\,d\nu = \int (f(x)-\lambda x) \bar{\mu}_{u}(dx)+\int (g(x)+\lambda x)\bar{\nu}_{v}(dx)+\lambda\left(\int x\bar{\mu}_{u}(dx)-\int x\bar{\nu}_{v}(dx) \right).
 \]

On the other hand, the duality for the standard Wasserstein distance is written as
\[
W_{p}^{p}(\bar{\mu}_{u},\bar{\nu}_{v})=\sup_{h,k}\left\{ \int h \,d\bar{\mu}_{u}+\int k\, d\bar{\nu}_{v}: h(x)+k(y)\le |x-y|^{p} \right\}. 
\]
For any $u,v$ and functions $f,g$ satisfying $f(x)+g(y)\le |x-y|^{p}+\lambda(x-y)$, the functions $h(x)=f(x)-\lambda x$ and $k(x)=g(x)+\lambda x$ satisfy $h(x)+k(y)\le |x-y|^{p}$.  Thus, 
\[
\int f\,d\mu+\int g\,d\nu \le  W_{p}^{p}(\bar{\mu}_{u},\bar{\nu}_{v})+\lambda\left(\int x\bar{\mu}_{u}(dx)-\int x\bar{\nu}_{v}(dx) \right)
 \]

Therefore, if $u,v$ are such that $\int x\,\bar{\mu}_{u}(dx)-\int x\,\bar{\nu}_{v}(dx)=0$ we obtain that
 \begin{equation}\label{es:2}
 \mathcal{W}_{p}^{p}(\mu,\nu)\ge\sup_{(\lambda,f,g)}\left\{ \int f\,d\mu+\int g\,d\nu:  f(x)+g(y)\le |x-y|^{p}+\lambda(x-y), \,\forall x,y\in\R \right\}.
 \end{equation}
 This still holds if there are no $u,v$ such that $\int x\,\bar{\mu}_{u}(dx)=\int x\,\bar{\nu}_{v}(dx)$ because in this case,  $\mathcal{W}_{p}(\mu,\nu)=+\infty$ but we do not need this case.    
 
 For the reverse inequality we mimic the proof of the Kantorovich duality as presented in \cite[Page 26]{Villani2}.   First, for $p\ge1$ and $\lambda\in\R$ we define the cost function
 \[
 c_{\lambda,p}(x,y)=\begin{cases} 
 |x-y|^{p}-\lambda(x-y)+(p-1)\left(\frac{|\lambda|}{p} \right)^{\frac{p}{p-1}} & p>1\text{ and any }\lambda\\ 
 |x-y|-\lambda(x-y) & p=1, \lambda\in(-1,1).
 \end{cases}  
 \]
 Notice that $c(x,y)\ge0$ and in the case $p=1$ we will restrict the values of $\lambda$ to the interval $(-1,1)$ and the difference from the cost function $|x-y|^{p}-\lambda(x-y)$ is just a constant.  Within this setting, the reverse of \eqref{es:2} becomes equivalent to proving that for fixed $\lambda\in\R$ (with $\lambda\in(-1,1)$ for $p=1$) and for any $u,v\in[0,2\pi]$, 
 \begin{equation}\label{es:1}
 \sup_{(f,g)}\left\{ \int f\,d\mu+\int g\,d\nu:  f(x)+g(y)\le c_{\lambda,p}(x,y), \,\forall x,y\in\R \right\} = \inf_{\pi\in \Pi(\bar{\mu}_{u},\bar{\nu}_{v})} \int c_{\lambda,p}(x,y)\pi(dx,\, dy), 
 \end{equation}
 where $\Pi(\zeta,\eta)$ is the set of probability measures on $\R\times \R$ with marginals $\zeta, \eta$.  
 
 To show \eqref{es:1} we use the Fenchel-Rockafeller duality Theorem (see \cite[Theorem 1.9]{Villani2}).  To do this we consider the space $E$ of continous and periodic functions on $\R\times\R$ endowed with the supremum norm.  This is the same as the space of continuous functions on $\T\times\T$ and its dual is the space of Borel measures with the total variation norm on $\T\times \T$ which can be identified with the space of Borel measures on $[u,u+2\pi)\times [v,v+2\pi)$.  The function $c_{\lambda,p}$ is certainly continuous on $\R\times \R$.  Now we introduce the functionals $\Theta:E\to\R\cup\{\infty\}$ and $\Xi:E^{*}\to \R\cup\{ \infty\}$, 
 \[
 \Theta(w)=\begin{cases}
 0 & \text{ if } w(x,y)\ge -c_{\lambda,p}(x,y) \\ 
 +\infty & \text{else}
 \end{cases}
 \text{ and }
 \Xi(w)=\begin{cases}
 \int f \,d\mu+\int g \,d\nu & \text{ if } w(x,y)=f(x)+g(y) \\
 +\infty & \text{otherwise}.
 \end{cases}
 \]
With these functionals, the rest of the proof of \eqref{es:1} follows exactly the same proof as in \cite[Pages 25-26]{Villani2} using the duality Theorem already mentioned.  
 
 Once we established \eqref{es:1}, we can simply prove the reverse inequality of \eqref{es:2}.   Indeed \eqref{es:1} combined with the fact that $c_{\lambda,p}(x,y)-|x-y|^{p}+\lambda(x-y)$ is constant,  we now can write for fixed $p\ge1$ and $\lambda\in\R$ (with $\lambda\in(-1,1)$ if $p=1$)
 \[
 \begin{split}
 \sup_{(\lambda,f,g)}&\left\{ \int f\,d\mu+\int g\,d\nu:  f(x)+g(y)\le |x-y|^{p}-\lambda(x-y), \,\forall x,y\in\R \right\} \\
 & \ge\inf_{\pi\in \Pi (\bar{\mu}_{u},\bar{\nu}_{v})} \int (|x-y|^{p}-\lambda(x-y))\pi(dx,\, dy).
 \end{split}
 \]
 If we pick $u,v$ such that $\int x\bar{\mu}_{u}(dx)=\int x\bar{\nu}_{v}(dx)$ (which is now possible because $\mathcal{W}_{p}(\mu,\nu)<\infty$), then we obtain that
\[
\begin{split}
\sup_{(\lambda,f,g)}&\left\{ \int f\,d\mu+\int g\,d\nu:  f(x)+g(y)\le |x-y|^{p}-\lambda(x-y), \,\forall x,y\in\R \right\} \\
& \ge\inf_{\pi\in \Pi(\bar{\mu}_{u},\bar{\nu}_{v})} \int |x-y|^{p}\pi(dx,\, dy)=W_{p}^{p}(\bar{\mu}_{u},\bar{\nu}_{v})
\end{split}
 \]
and upon taking the supremum over such $u,v$ gives the conclusion that 
\[
\sup_{(\lambda,f,g)}\left\{ \int f\,d\mu+\int g\,d\nu:  f(x)+g(y)\le |x-y|^{p}-\lambda(x-y), \,\forall x,y\in\R \right\} \ge \mathcal{W}_{p}^{p}(\mu,\nu)
\]
which finishes the proof.

\item Set
\[
(U_{t}f)(x)=\inf\left\{ f(y)+\frac{(x-y)^{2}}{t} \right\}.
\]
It is known from the Hopf-Lax formula (\cite[Theorem 4 and 5, Section 3.3]{Evans}) that the infimum convolution semigroup satisfies for $\tilde{f}_{t}=U_{t}f$, 
\[
\partial_{t}\tilde{f}_{t}=-\frac{1}{4}(\tilde{f}_{t}')^{2} \text{ with }\tilde{f}_{0}(x)=f(x).
\]
Now it becomes a simple exercise to check that 
\[
(\mathcal{U}_{t}^{\lambda}f)(x)=\lambda x+(U_{t}f_{\lambda})(x) \text{ with }f_{\lambda}(x)=f(x)-\lambda x.
\]
From this, the rest follows easily.  
\item We just  have to observe that for a given function $f$, the best function $g$ such that $g(x)-f(y)\le (x-y)^{2}+\lambda(x-y)$ is given by $g=\mathcal{U}^{\lambda}_{1}f$.  The rest is a consequence of the first part of the Theorem.

\item  From the first part of the Theorem, 
\[
\mathcal{W}_{1}(\mu,\nu)=\sup_{\lambda,f,g}\left\{ \int fd\mu-\int g d\nu: f(x)\le g(y)+|x-y|+\lambda(x-y) \right\}.
\]
Given a $\lambda\in\R$ and a function $g$, the best choice of $f$ is determined by 
\[
f(x)=\inf_{y}\{ g(y)+|x-y|+\lambda(x-y) \}. 
\]  
In particular, $f(x)\le g(x)$ and 
\[
\int f d\mu-\int gd\nu\le \int f d(\mu-\nu)
\] 
Denoting now $f(x)=\inf_{y} \{g(y)+|x-y|+\lambda(x-y)\}$, for every $\epsilon>0$ we can find a $y_{\epsilon}$ such that 
\[
f(x)-f(x')\le \epsilon+g(y_{\epsilon})+|x-y_{\epsilon}|+\lambda(x-y_{\epsilon})-(g(y_{\epsilon})+|x'-y_{\epsilon}|+\lambda(x'-y_{\epsilon}))\le \epsilon+ |x-x'|+\lambda(x-x').
\] 
Therefore we get that 
\[
|f(x)-f(x')-\lambda(x-x')|\le |x-x'| \text{ for all }x,x'.  
\] 
Hence, 
\[
\mathcal{W}_{1}(\mu,\nu)\le \sup_{\lambda,f}\left\{ \int f\,d(\mu-\nu): |f(x)-f(y)-\lambda(x-y)|\le |x-y| \right\}.
\]
On the other hand, if $f$ is as above (namely $|f(x)-f(y)-\lambda(x-y)|\le |x-y|$), take $g(x)=\inf_{y}\{f(y)+|x-y|+\lambda(x-y)\}$ and observe now that $f(x)= g(x)$, and moreover $f(x)-g(y)\le |x-y|+\lambda(x-y)$, from which
\[
\int fd\mu-\int fd\nu= \int fd\mu-\int gd\nu. 
\]
This is now enough to conclude that 
\[
\mathcal{W}_{1}(\mu,\nu)\ge \sup_{\lambda,f}\left\{ \int f\,d(\mu-\nu): |f(x)-f(y)-\lambda(x-y)|\le |x-y| \right\}
\]
which ends  the proof.  \qedhere

\end{enumerate}
\end{proof}

\section{Transportation, Log-Sobolev and HWI inequalities}\label{S:8}

\subsection{Transportation Inequality}

In this section we discuss the transportation inequality.  The approach is inspired by \cite{LP} and \cite{IP}.   

\begin{definition} We say that a potential $Q$ on $\T$ satisfies the transportation inequality if there is a positive $\rho>0$, such that for any other smooth measure $\mu$ on $\T$, 
\begin{equation}\label{e1:1t}
\rho \mathcal{W}_{2}^{2}(\mu,\mu_{Q})\le E_{Q}(\mu)-E_{Q}(\mu_{Q})
\end{equation}
where $\mathcal{W}_{2}(\mu,\nu)$ is the distance defined in \eqref{e:W}. 

We refer to $\eqref{e1:1t}$ as $T(\rho)$.  
\end{definition}

\begin{theorem}\label{t:4}
Any $C^{3}$ potential  $Q:\mathbf{T}\to\R$  such that $Q''\ge \rho -1/2$ for some $\rho>0$ satisfies $T(\rho/2)$.   

\end{theorem}

\begin{proof}  For any probability measure $\mu$ on $\T$ with finite energy $E_{Q}(\mu)$, it is clear that $\mu$ does not have atoms.  We may also assume that the measure $\mu$ has a smooth density, otherwise we can use careful approximations as in \cite{HPU1}.   Thus according to (6) in Proposition~\ref{p:W} we can choose $t$ and $\theta=\theta_{t}:[t,t+2\pi) \to [t,t+2\pi)$ such that $\theta$ is non-decreasing and 
\[
\mathcal{W}_{2}^2(\mu,\mu_{Q})=W_{2}^2(\bar{\mu}_{t},\bar{\mu}_{Q,t}) = \int (x-\theta(x))^{2}\bar{\mu}_{Q,t}(dx).
\]
The transportation inequality becomes now (recall that $V(x)=Q(e^{ix})$)
\[
\frac{\rho}{2} \int(\theta(x)-x)^{2}\bar{\mu}_{Q,t}(dx)\le \int (V(\theta(x))-V(x))\bar{\mu}_{Q,t}(dx)-\iint\log\left(\frac{\sin((\theta(x)-\theta(y))/2)}{\sin((x-y)/2)}\right)\bar{\mu}_{Q,t}(dx)\bar{\mu}_{Q,t}(dy).
\]
Since $V''(x)=Q''(e^{ix})\ge (\rho-1/2)$ for $x\in [t,t+2\pi)$ we get that 
\[
V(\theta(x))-V(x)\ge V'(x)(\theta(x)-x)+(\rho/2-1/4)(\theta(x)-x)^{2}
\]
and thus it suffices to prove that 
\begin{equation}\label{e:t2b}
\frac{1}{4}\int (\theta(x)-x)^{2}\bar{\mu}_{Q,t}(dx)\le \int V'(x)(\theta(x)-x)\bar{\mu}_{Q,t}(dx) -\iint\log\left(\frac{\sin((\theta(x)-\theta(y))/2)}{\sin((x-y)/2)}\right)\bar{\mu}_{Q,t}(dx)\bar{\mu}_{Q,t}(dy).
\end{equation}
At this point we can use the variational characterization of the minimizer $\mu_{Q}$ to justify that 
\[
V'(x)=\frac{d}{dx}2\int\log |\sin((x-y)/2)|\bar{\mu}_{Q,t}(dx)=\int\cot((x-y)/2)\bar{\mu}_{Q,t}(dx)
\]
for $x$ in $[t,t+2\pi)$ and the integral above has to be taken in the sense of a principal value integral.  
By Corollary~\ref{c:2} we know that the support of $\bar{\mu}_{Q,t}$ is the whole interval $[t,t+2\pi)$.   From here, and the fact that the means of both $\bar{\mu}_{Q,t}$ and $\bar{\mu}_{t}$ are equal, we obtain the following crucial identities
\begin{align*}
\int (\theta(x)-x)^{2}\bar{\mu}_{Q,t}(dx) & = \frac{1}{2}\iint((\theta(x)-x)-(\theta(y)-y))^{2}\bar{\mu}_{Q,t}(dx)\bar{\mu}_{Q,t}(dy)\\
\int V'(x)(\theta(x)-x)\bar{\mu}_{Q,t}(dx) & =\frac{1}{2}\iint((\theta(x)-x)-(\theta(y)-y))\cot((x-y)/2)\bar{\mu}_{Q,t}(dx)\bar{\mu}_{Q,t}(dy).
\end{align*}
Further, since  $\frac{1}{2}x^{2}+\log(\sin(x))$ is concave on $[0,\pi)$, we have that for $a,b\in[0,\pi)$
\begin{equation}\label{1}
\frac{1}{2}(a-b)^{2}\le (a-b)\cot(b)-\log\left( \frac{\sin(a)}{\sin(b)} \right)
\end{equation}
and this with $a=(\theta(x)-\theta(y))/2$, $b=(x-y)/2$, $x>y$, completes the proof.  \qedhere
\end{proof}

In the case of measures on the real line, the transport inequalities from \cite{LP} for some cases proved to be sharp.  In the present context, the above proof does not give the sharp constants in the case of vanishing potential $Q$ as it is proved in the following.  It might be true that it could be sharp for some other potentials $Q$, but that we do not know.

\begin{proposition}\label{p:6}  For the case of $Q\equiv 0$, $\rho=1/2$, there exists $\epsilon>0$ such that $T((1+\epsilon)/4)$ holds true.  Thus $T(1/4)$  is not sharp for $Q\equiv0$.   

\end{proposition}

\begin{proof}
For $Q\equiv 0$ we know that $\rho=1/2$ and $\mu_{Q}=\alpha$.   In order to get a sharp inequality, we need to saturate \eqref{1} and that is the case for either $a$ close to $b$ or $b$ close to $\pi/2$.  

An improvement of the inequality \eqref{1} would be of the form
\begin{equation}\label{e:t1cv}
\frac{(a-b)^2 h(b)}{2} \le (a-b)\cot(b)-\log\left( \frac{\sin(a)}{\sin(b)} \right)
\end{equation}
for some function $h(b)$.  Clearly, \eqref{1} gives that we can take $h(b)=1$ for any $b$. The idea is that replacing  $a=(\theta(x)-\theta(y))/2$, $b=(x-y)/2$, $x>y$ and integrating with respect to $x$ and $y$, we obtain that 
\begin{equation}\label{e:t5}
\frac{1}{8}\iint_{-\pi}^{\pi}(\psi(x)-\psi(y))^2 h(|x-y|/2)\,dx\,dy \le -\int_{-\pi}^{\pi}\int_{-\pi}^{\pi}\log\left(\frac{\sin((\theta(x)-\theta(y))/2)}{\sin((x-y)/2)}\right)\,dx\, dy
\end{equation}
with the notation $\psi(x)=\theta(x)-x$.   

The best choice in \eqref{e:t1cv} is the following function
\begin{equation}\label{e:t3}
h(b)=\inf_{a\in[0,\pi]}\frac{2\left( (a-b)\cot(b)-\log\left( \frac{\sin(a)}{\sin(b)} \right)\right)}{(a-b)^{2}}.
\end{equation}
Thus, we try to prove an inequality of the form 
\begin{equation}\label{e:t4bb}
c\int_{-\pi}^{\pi}\psi^{2}(x)\,dx\le \int_{-\pi}^{\pi}\int_{-\pi}^{\pi}(\psi(x)-\psi(y))^2 h(|x-y|/2)\,dx\,dy.
\end{equation}
Notice that we clearly have $h(b)\ge 1$ and thus the above inequality holds with $c=2$ which yields $T(1/4)$.  Nevertheless, the point is that  the above inequality may give a better constant than $c=2$.   

To understand the function $h$, notice that from the Taylor's formula with the integral remainder, we learn that 
\[
\left( (a-b)\cot(b)-\log\left( \frac{\sin(a)}{\sin(b)} \right)\right) = (a-b)^{2}\int_{0}^{1}\frac{1-s}{\sin^{2}(b+s(a-b))}ds
\]
so 
\[
h(b)=\inf_{a\in[0,\pi]}\int_{0}^{1}\frac{2(1-s)}{\sin^{2}(b+s(a-b))}ds.  
\]
It is clear from this representation that $h$ is a continuous function on $(0,\pi)$ and in addition to this, $h(b)\ge1$ with equality only for the case of $b=\pi/2$.   Moreover $h(0)=h(\pi)=+\infty$.   Therefore, we can take a continuous  function $g:[-2\pi,2\pi]\to\R$ such that $g(b)=g(-b)$, $h(|b|/2)\ge g(b)$ and $g(b)\ge1$ for $b\ge0$ with equality only for $b=\pm\pi$.   

With this choice we claim that there is a constant $\delta>0$ such that for any continuous function $\psi$ with $\int_{-\pi}^{\pi}\psi(x)dx=0$ we have
\begin{equation}\label{e:t6b}
2(1+\delta) \int\psi^{2}(x)\,dx\le \iint (\psi(x)-\psi(y))^2 g(|x-y|)\,dx\,dy.
\end{equation}
This in turn will imply that the constant $c$ in \eqref{e:t4bb} satisfies $c>2$, thus combining this with \eqref{e:t5} we obtain that the transportation inequality is $T(c/8)$ with $c/8>1/4$.   

To see why \eqref{e:t6b} is fulfilled, we follow an argument suggested to the author by Michael Loss.  It is easy to see first that $g(t)=\sum_{n} e^{-in t}\hat{g}(n)$ where for an arbitrary function $\phi:[-\pi,\pi]\to\R$ we denote $\hat{\phi}(n)=\frac{1}{2\pi}\int_{-\pi}^{\pi}\phi(t)e^{in t}dx$.  Notice that because $g$ is an even function, the Fourier coefficients $\hat{g}(n)$ are all real for any integer $n\ge0$. After some elementary manipulations with Fourier series, the above inequality translates into 
\[
(1+\delta)\sum_{n\ne 0}|\hat{\psi}(n)|^{2}\le \sum_{n}|\hat{\psi}(n)|^{2}(\hat{g}(0)-\hat{g}(n))
\]
which follows once we prove that for some $\delta$ and any $n\ge0$,
\[
1+\delta \le \hat{g}(0)-\hat{g}(n)=\frac{1}{2\pi}\int_{-\pi}^{\pi}(1-\cos(nx))g(x)dx=\frac{1}{\pi}\int_{0}^{\pi}(1-\cos(nx))g(x)dx.
\]
In the first place, for any finite number of values of $n$, it is clear that 
\[
1=\frac{1}{2\pi}\int_{-\pi}^{\pi}(1-\cos(nx))dx=\frac{1}{\pi}\int_{0}^{\pi}(1-\cos(nx))dx<\frac{1}{\pi}\int_{0}^{\pi}(1-\cos(nx))g(x)dx.
\]
On the other hand, for large $n$, 
\[
\frac{1}{2\pi}\int_{-\pi}^{\pi}(1-\cos(nx))g(x)dx=\frac{1}{\pi}\int_{0}^{\pi}(1-\cos(nx))g(x)dx=\int_{0}^{\pi}(1-\cos(y))\frac{1}{n\pi}\sum_{k=0}^{n-1} g\left(\frac{k\pi+y}{n} \right)dy
\]
which converges as $n\to\infty$ to 
\[
\frac{1}{\pi}\int_{0}^{1}g(t \pi)dt=\frac{1}{\pi}\int_{0}^{\pi}g(s)ds>1.  
\]
Consequently, we can choose a small $\delta>0$ such that for any $n$, 
\[
1+\delta \le \hat{g}(0)-\hat{g}(n)=\frac{1}{2\pi}\int_{-\pi}^{\pi}(1-\cos(nx))g(x)dx,
\]
which ends the proof.  \qedhere
\end{proof}

\begin{conjecture}

For $Q\equiv 0$, $T(1/2)$ holds true.

\end{conjecture}

Here is some supporting evidence for this claim.   If we take the function $h(y)=1/\sin^2(y)$  then we have the following sequence of equalities
\[
\frac{1}{8}\int (\psi(x)-\psi(y))^2 h(|x-y|)\,dx\,dy=\frac{1}{8}\int\frac{(\psi(x)-\psi(y))^2}{\sin^2((x-y)/2)}\,dx\,dy=\frac{1}{2}\int\frac{(\phi(z)-\phi(w))^2}{|z-w|^2}\alpha(dz)\alpha(dw)
\]
where $\phi(e^{ix})=\psi(x)$.  The last term is precisely $\langle\mathcal{N}\phi,\phi\rangle$ and because $\int \phi \,d\alpha=0$, we know from the fact that the spectrum of $\mathcal{N}$ is $0,\pm1,\pm2,\dots,$ that $\langle\mathcal{N}\phi,\phi\rangle\ge \int \phi^2\,d\alpha$ and thus 
\[
\frac{1}{2}\int \psi^2(x)\,dx \le \frac{1}{8}\int(\psi(x)-\psi(y))^2 h(|x-y|)\,dx\,dy.
\]

Thus, if \eqref{e:t1cv} were true with $h(b)=1/\sin^2(b)$, then we would obtain that 
\[
\frac{1}{2}\int (\theta(x)-x)^2(x)dx\le -\iint\log\left(\frac{\sin((\theta(x)-\theta(y))/2)}{\sin((x-y)/2)}\right)\,dx\, dy
\]
which is in fact $T(1/2)$, thus getting the claim of the Conjecture.      

Unfortunately, \eqref{e:t1cv} is not satisfied with $h(b)=1/\sin^2(b)$, thus the argument outlined above does not work.  However, it might be that the operator defined by the Dirichlet form of the left hand side of \eqref{e:t5} with the function $h$ from  \eqref{e:t3} has the same spectral gap as the operator $\mathcal{N}$, though, I do not know how to show that.

\subsection{Log-Sobolev Inequality}

This section deals with the following form of Log-Sobolev inequality which was introduced in \cite{HPU2}.  

\begin{definition}
We say that the $C^{1}$ potential $Q$ satisfies the free Log-Sobolev if there exists a positive $\rho>0$ such that  for any other smooth measure $\mu$ on $\T$
\begin{equation}\label{e4:1}
E_{Q}(\mu)-E_{Q}(\mu_{Q})\le \frac{1}{4\rho}I_{Q}(\mu)
\end{equation}
with 
\[
I_{Q}(\mu):=\int(H\mu-Q')^{2}d\mu-\left(\int Q'd\mu\right)^{2}\text{ and } H\mu(x)=-p.v. \int \frac{z+w}{z-w}\mu(dw)=p.v.\int \cot\left(\frac{x-y}{2}\right)\bar{\mu}_{u}(dy)
\]
taken in the sense of principal value and for any $u\in[0,2\pi)$.  

We will refer in the sequel to  \eqref{e4:1} as $LSI(\rho)$. 
\end{definition}

On the real line, the Fisher information as introduced by Voiculescu is given by $I_{Q}(\mu):=\int(H\mu-Q')^{2}d\mu$ which is clearly non-negative.  Here, on the circle we subtract the factor $\left(\int Q'd\mu\right)^{2}$ and thus it is not clear if $I_{Q}$ remains non-negative.   To see that this is indeed so, for a smooth measure $\mu$, one can easily see that $\int (H\mu) d\mu=0$ and from this we can rewrite  
\[
I_{Q}(\mu)=\int\left((H\mu-Q') - \int (H\mu-Q')d\mu\right)^{2}d\mu
\]
which is obviously non-negative.  If $\mu_{Q}$ has full support, then equality is attained for  $\mu=\mu_{Q}$ as it is observed in \cite{HPU2} and one can also see from the variational characterization \eqref{ep:var} or for instance from \eqref{e1:03}.  

The main result of this section is essentially due to Hiai Petz and Ueda from  \cite{HPU2}, though they use  random matrices in their approach.  Our proof is different. 

\begin{theorem}\label{t:lsi}
Any $C^{3}$ potential $Q:\mathbf{T}\to\R$ such that $Q''\ge \rho-1/2$ for some $\rho>0$ satisfies $LSI(\rho/2)$.   

\end{theorem}

\begin{proof}
We proceed with the same arguments and notations as for the transportation inequality.  Thus the map $\theta$ transports $\bar{\mu}_{Q,t}$ into $\bar{\mu}_{t}$ and with this we write the inequality in the equivalent form as
\begin{align*}
\iint\log\left(\frac{\sin((x-y)/2)}{\sin((\theta(x)-\theta(y))/2)}\right)&\bar{\mu}_{Q,t}(dx)\bar{\mu}_{Q,t}(dy) \le \int (V(x)-V(\theta(x)))\bar{\mu}_{Q,t}(dx) \\ &+\frac{1}{2\rho}\int(H\mu(\theta(x))-V'(\theta(x)))^{2}\bar{\mu}_{Q,t}(dx)-\frac{1}{2\rho}\left(\int V'(\theta(x))\bar{\mu}_{Q,t}(dx) \right)^{2}.
\end{align*}

As a word of caution, we may assume that the map $\theta$ is a smooth map, otherwise we can rely on careful approximations arguments outlined in \cite{HPU1} or \cite{HPU2}.  

Now, because $V(x)-(\rho/2-1/4) x^{2}$ is convex,
\[
V(x)-V(\theta(x))\ge (\rho-1/2)(x-\theta(x))^{2}/2+(x-\theta(x))V'(\theta(x))
\]
and then 
\begin{align*}
V(x)&-V(\theta(x))+\frac{1}{2\rho}(H\mu(\theta(x))-V'(\theta(x)))^{2}\ge \\ 
&\quad \frac{2\rho-1}{4}(\theta(x)-x)^{2}+(x-\theta(x))V'(\theta(x)) +\frac{1}{2\rho}(H\mu(\theta(x))-V'(\theta(x)))^{2}\\ 
& =- (1/4)(\theta(x)-x)^{2}+ (x-\theta(x))H\mu(\theta(x))+ \frac{1}{2\rho}\left[\rho(x-\theta(x)) + V'(\theta(x))-H\mu(\theta(x))  \right]^{2}
\end{align*}
Integrating with respect to $\bar{\mu}_{Q,t}$, using the Cauchy-Schwarz inequality and the fact that 
\[
\int(x-\theta(x))\bar{\mu}_{Q,t}(dx)=\int H\mu(\theta(x))\bar{\mu}_{Q,t}(dx)=0,
\]
 results with
\[
\int \left[\rho(x-\theta(x)) + V'(\theta(x))-H\mu(\theta(x))  \right]^{2}\mu_{Q}(dx)\ge \left(\int V'(\theta(x))\mu_{Q}(dx)\right)^{2}.
\]
The required inequality follows thus from
\begin{equation}\label{2}
\iint\log\left(\frac{\sin((x-y)/2)}{\sin((\theta(x)-\theta(y))/2)}\right)\bar{\mu}_{Q,t}(dx)\bar{\mu}_{Q,t}(dy) \le -(1/4)\int (\theta(x)-x)^{2}\bar{\mu}_{Q,t}(dx)+\int (x-\theta(x))H\mu(\theta(x))\bar{\mu}_{Q,t}(dx).
\end{equation}
Furthermore, since
\begin{align*}
\int (\theta(x)-x)^{2}\bar{\mu}_{Q,t}(dx) & = \frac{1}{2}\iint ((\theta(x)-x)-(\theta(y)-y))^{2}\bar{\mu}_{Q,t}(dx)\bar{\mu}_{Q,t}(dy) \\
\int (x-\theta(x))H\mu(\theta(x))\bar{\mu}_{Q,t}(dx)& = \frac{1}{2}\iint ((x-\theta(x))-(y-\theta(y)))\cot((\theta(x)-\theta(y))/2)\bar{\mu}_{Q,t}(dx)\bar{\mu}_{Q,t}(dy) 
\end{align*}
the inequality \eqref{2} becomes a consequence of \eqref{1} with $a=(x-y)/2$ and $b=(\theta(x)-\theta(y))/2$. 
\end{proof}

As in the case of the transportation, we do not think $LSI(\rho/2)$ is sharp.   In particular, for $Q\equiv 0$, we can prove as in the case of the transportation the following statement.  

\begin{proposition}
For the case $Q\equiv 0$, there exists $\delta>0$ such that $LSI((1+\delta)/4)$ is satisfied.     
\end{proposition}

A possible proof runs similarly to the one for the transportation inequality.  However a more elegant approach is simply a consequence of the HWI inequality below, particularly it follows from the second part of Corollary~\ref{c:3} since the transportation constant in the case $Q\equiv0$ is strictly greater than $1/4$.  

As in the transportation inequality case, we believe that for the case of $Q\equiv 0$, the constant in Theorem~\ref{t:lsi} is $1/2$ not $1/4$.   

\begin{conjecture}
For the case $Q\equiv0$, $LSI(1/2)$ is true.   
\end{conjecture}

\subsection{HWI inequality}

\begin{definition}

We say that the potential $Q$ satisfies the free $HWI(\rho)$ for $\rho\in\R$ if for any smooth measure $\mu$ on $\T$
\begin{equation}\label{e5:1}
E_{Q}(\mu)-E_{Q}(\mu_{Q})\le \sqrt{I_{Q}(\mu)}\mathcal{W}_{2}(\mu,\mu_{Q})-\rho \mathcal{W}_{2}(\mu,\mu_{Q})^{2}.
\end{equation}
\end{definition}

With the same notations as above we have the following result.  

\begin{theorem}  Any $C^{3}$ potential $Q:\mathbf{T}\to\R$  such that $Q''\ge \rho-1/2$ for some  $\rho\in\R$ satisfies the free $HWI(\rho/2)$.
\end{theorem}

\begin{proof} We use the same notations as in the proof of the transportation inequality with the addition that $\theta$ is assumed to be smooth.  

After some elementary rearrangements, we need to prove that
\begin{align*}
\iint\log&\left(\frac{\sin((x-y)/2)}{\sin((\theta(x)-\theta(y))/2)}\right)\bar{\mu}_{Q,t}(dx)\bar{\mu}_{Q,t}(dy) \le \int (V(x)-V(\theta(x)))\bar{\mu}_{Q,t}(dx) \\ &+ \left( \int(x-\theta(x))^{2}\bar{\mu}_{Q,t}(dx)\right)^{1/2}\left(\int(H\mu(\theta(x))-V'(\theta(x)))^{2}\bar{\mu}_{Q,t}(dx)-\left(\int V'(\theta(x))\bar{\mu}_{Q,t}(dx) \right)^{2}\right)^{1/2}\\ & -\frac{\rho}{2}\int(x-\theta(x))^{2}\bar{\mu}_{Q,t}(dx).
\end{align*}
Since
\[
V(x)-V(\theta(x))\ge(\rho/2-1/4)(x-\theta(x))^{2}+(x-\theta(x))V'(\theta(x)),
\]
and \eqref{2}, it suffices to show that
\begin{align*}
\int(x-\theta(x))(H\mu(\theta(x))-V'(\theta(x))&\bar{\mu}_{Q,t}(dx)\le \left( \int(x-\theta(x))^{2}\bar{\mu}_{Q,t}(dx)\right)^{1/2}\times \\ &\left(\int(H\mu(\theta(x))-V'(\theta(x)))^{2}\bar{\mu}_{Q,t}(dx)-\left(\int V'(\theta(x))\bar{\mu}_{Q,t}(dx) \right)^{2}\right)^{1/2}.
\end{align*}
Let $f(x)=x-\theta(x)$, $g(x)=H\mu(\theta(x))-V'(\theta(x))$ and $h(x)=-V'(\theta(x))$.  Then we have that $\int f(x)\bar{\mu}_{Q,t}(dx)=0$, $\int g(x)\bar{\mu}_{Q,t}(dx)=\int h(x)\bar{\mu}_{Q,t}(dx)$, and the rest follows from 
\begin{align*}
\int f(x)g(x)\bar{\mu}_{Q,t}(dx)=&\int f(x)\left(g(x)-\int h(y)\bar{\mu}_{Q,t}(dy)\right)\bar{\mu}_{Q,t}(dx) \\ 
&\le \left(\int f^{2}(x)\bar{\mu}_{Q,t}(dx) \right)^{1/2}\left( \int \left(g(x)-\int h(y)\bar{\mu}_{Q,t}(dy) \right)^{2} \bar{\mu}_{Q,t}(dx)\right)^{1/2} 
\\ & =\left(\int f^{2}(x)\bar{\mu}_{Q,t}(dx) \right)^{1/2}\left( \int g^{2}(x)\bar{\mu}_{Q,t}(dx)-\left(\int h(y)\bar{\mu}_{Q,t}(dy)\right)^{2} \right)^{1/2}.\qedhere
\end{align*}
\end{proof}

\begin{remark}  As in the case of Log-Sobolev inequality and transportation, for $Q\equiv 0$, the $HWI(1/4)$ is probably not sharp.  We conjecture as well that the sharp inequality is $HWI(1/2)$, but we do not have a proof.  

\end{remark}

\subsection{Brunn-Minkowski inequality}

The result in this section is the following.  
\begin{theorem}
Assume that $Q_{1},Q_{2},Q_{3}$ are continuous potentials on $\T$ such that for some $a\in(0,1)$,
\begin{equation}\label{e:17}
a Q_{1}(e^{ix})+(1-a)Q_{2}(e^{iy})\ge Q_{3}(e^{i(ax+(1-a)y)})\quad\text{for all}\quad x,y\in\R,
\end{equation}
Then 
\begin{equation}\label{e:18'}
aE_{Q_{1}}(\mu_{Q_{1}})+(1-a)E_{Q_{2}}(\mu_{Q_{2}})\ge E_{Q_{3}}(\mu_{Q_{3}}).
\end{equation}
\end{theorem}

\begin{proof}  Map all measures onto $[0,2\pi)$ and then take the (non-decreasing) transportation map $\theta$ from $\nu_{1}=\bar{\mu}_{Q_{1},0}$ into $\nu_{2}=\bar{\mu}_{Q_{2},0}$.   The existence of $\theta$ is guaranteed by the fact that $\mu_{Q_{1}}$ does not have atoms. 

Now, we obviously have
\[
\iint\log|e^{ix}-e^{iy}|\mu(dx)\mu(dy)=2\iint_{x>y}\log(2\sin((x-y)/2))\mu(dx)\mu(dy).
\]  
 Using this we argue that    
\begin{align*}
\int & (aV_{1}(x)+(1-a)V_{2}(\theta(x)))\nu_{1}(dx) \\ 
&\qquad-2\iint_{x>y} \left\{a\log\left(2\sin\left(\frac{x-y}{2}\right)\right)+(1-a)\log\left(2\sin\left(\frac{\theta(x)-\theta(y)}{2} \right)\right)\right\}\nu_{1}(dx)\nu_{1}(dy)\\ 
&\ge \int V_{3}\big (ax+(1-a)\theta(x) \big )\nu_{1}(dx)
 - 2\iint_{x>y} \log\left(2\sin\left(\frac{((ax+(1-a)\theta(x))-(ay+(1-a)\theta(y))}{2}\right)\right)\nu_{1}(dx)\nu_{1}(dy) \\ 
& = E_{V_{3}}(\nu) \ge E_{Q_{3}}(\mu_{Q_{3}})
\end{align*}
where $\nu=(a {\rm{id}}+(1-a)\theta)_{\#}\nu_{1}$ and we used  the concavity of $\log\sin(t)$ on $(0,\pi)$. The proof is complete. \qedhere

\end{proof}

\section{Hierarchical Implications}\label{S:9}

In this section we study the implications between the inequalities introduced.  This is very similar to the classical inequalities.  

\subsection{HWI and Log-Sobolev}
The following is the analog to the classical result of Otto-Villani from \cite{OV}.  

\begin{corollary}\label{c:3}
\begin{enumerate}
\item For $\rho>0$, $HWI(\rho)$ implies $LSI(\rho)$.  

\item If $HWI(\rho)$ holds for some $\rho\in\R$, then $T(C)$ with $C>\max\{0,-\rho\}$ implies $LSI(K)$ inequality with constant 
\[
K=\begin{cases}
\frac{(C+\rho)^{2}}{4C} & \rho\le 0 \\
\rho & \rho>C>0\\ 
\frac{(C+\rho)^{2}}{4C} & 0<\rho\le C.
\end{cases}
\]  In particular, if $\rho>0$ and $C>\rho$, the constant $K>\rho$.  
\end{enumerate}
\end{corollary}

\begin{proof}
\begin{enumerate}
\item The conclusion follows from $ab-\rho a^{2}\le b^{2}/(4\rho)$ for any real numbers $a$ and $b$ and $\rho>0$.  
\item $HWI(\rho)$ and $T(C)$ yields that for any $\delta>0$, such that $\delta>\rho$,
\[\tag{*}
\begin{split}
E_{Q}(\mu)-E_{Q}(\mu_{Q})&\le  \mathcal{W}_{2}(\mu,\mu_{Q})\sqrt{I(\mu)}-\rho \, \mathcal{W}_{2}^{2}(\mu,\mu_{Q}) \\
& \le \frac{1}{4\delta} I(\mu)+\delta\mathcal{W}_{2}^{2}(\mu,\mu_{Q})-\rho \, \mathcal{W}_{2}^{2}(\mu,\mu_{Q}) \\
&\le  \frac{1}{4\delta} I(\mu)+\frac{\delta-\rho}{C} \big (E_{Q}(\mu)-E_{Q}(\mu_{Q}) \big ).
\end{split}
\]
In the case, $\rho\le 0$, for any $\delta<C+\rho$ we arrive at
\[\tag{**}
E(\mu)-E(\mu_{V})\le \frac{C}{4\delta(C+\rho-\delta)} \, I(\mu).
\]
The minimum over $\delta<C+\rho$ is attained at $\delta=\frac{C+\rho}{2}$ and this yields $LSI(\frac{(C+\rho)^{2}}{4C})$.    

In the case $\rho>0$, we have from the first part that $LSI(\rho)$ holds true.   On the other hand, to have (**) we need to ensure that $\rho<\delta<C+\rho$.  Now we need to maximize  $\delta(C+\rho-\delta)$ over $\delta$,  subject to the constraints $\rho<\delta<C+\rho$.  The unconstrained maximum is attained at $(C+\rho)/2$.  This is in the interval $[\rho,C+\rho]$ if $C>\rho$.  Otherwise the maximum is attained at $\rho$.  Therefore  the conclusion.    

\end{enumerate}
\end{proof}

\subsection{Log-Sobolev implies Transportation}

In this section we show how the Log-Sobolev implies the transportation inequality.   We use here the main idea from \cite{L}.   

\begin{theorem}
For a $C^{1}$ potential $Q$ and a positive $\rho>0$, $LSI(\rho)$ implies $T(\rho)$.  
\end{theorem}

\begin{proof}  Take a probability measure $\mu$ on $\T$ such that $E_{Q}(\mu)<\infty$.  Thus the measure $\mu_{Q}$ does not have atoms and, according to Proposition~\ref{p:W}, we can find a $u\in[0,2\pi]$ such that there are lifts $\mu_{u}$ and $\mu_{Q,u}$ of $\mu$ and $\mu_{Q}$ to $[u,u+2\pi)$ such that they have the same mean.   

Now, according to Theorem~\ref{t:W}, we can represent the Wasserstein distance in the dual formulation and we will prove that the Log-Sobolev inequality implies the transportation inequality.

For a $\lambda\in\R$ and a periodic smooth function $g:\R\to\R$ we set $g_{t}(x)=\mathcal{U}^{\lambda}_{t}g$.  We have that 
\begin{equation}\label{e5b:1}
\partial_{t}g_{t}=-\frac{1}{4}(g_{t}'-\lambda)^{2} \text{ with }g_{0}(x)=g(x).  
\end{equation}

Now assume that the Log-Sobolev inequality \eqref{e4:1} holds for some constant $\rho>0$.  

For a given function $\phi$ on the circle (interpreted as a periodic function on the real line), define
\[
j_{Q}(\phi)=E_{Q}-E_{Q-\phi}
\]
and notice that 
\begin{equation}\label{e5b:1b}
j_{Q}(\phi)\ge E_{Q}-E_{Q-\phi}(\mu)=\int \phi\, d\mu-(E_{Q}(\mu)-E_{Q})
\end{equation}
for any measure $\mu$.  

Now, for $a>0$,  take 
\[
h_{t}=(a+\rho t)g_{t}-j_{t}, \text{ with } j_{t}=j_{Q}((a+\rho t)g_{t})
\]
and consider $\nu_{t}$ the equilibrium measure of $Q-h_{t}$.   Observe that this is well defined due to the fact that $h_{t}$ is a periodic function, thus a smooth function on the circle.  Notice that $j_{Q}(h_{t})=0$ and from this and \eqref{e5b:1b} we obtain that, 
\[
E_{Q}(\nu_{t})-E_{Q}\ge \int h_{s}d\nu_{t}
\]
for any $s\ge0$ and $t\ge0$ with equality for $s=t$.  This yields that $\int \partial_{t}h_{t}\,d\nu_{t}=0$.   

Now, applying $LSI(\rho)$ to $\nu_{t}$ and keeping in mind that $H\nu_{t}=Q'-h_{t}'$ (see \cite{ST}) on the support of $\nu_{t}$, we obtain that 
\[
   4\rho \int h_{t}\, d\nu_{t}\le \int (h_{t}')^{2}d\nu_{t}-\left( \int Q' d\nu_{t}\right)^{2}
\]
which upon using \eqref{e5b:1} gives 
\[
4\rho (a+\rho t)\int g_{t}\,d\nu_{t}-4\rho j_{t}\le -(a+\rho t)^{2}\left(4\int\partial_{t}g_{t}\,d\nu_{t}-2\lambda\int g_{t}'d\nu_{t}+\lambda^{2} \right)-\left( \int Q' d\nu_{t}\right)^{2}.
\]
Since $\int \partial_{t}h_{t}\,d\nu_{t}=0$, we can continue with 
\[
4\partial_{t}\left( \frac{j_{t}}{a+\rho t}\right) \le 2\lambda\int g_{t}'\,d\nu_{t}-\lambda^{2}-\frac{1}{(a+\rho t)^{2}}\left( \int Q' d\nu_{t}\right)^{2}. 
\]
We observe now that because $Q$ and $h_{t}$ are $C^{1}$, combined with \eqref{e1:03}, we get that 
\[
\int (Q'-h_{t}')d\nu_{t}=0.
\]
Written differently, we obtain that $\int Q'd\nu_{t}=(a+\rho t)\int g_{t}'\,d\nu_{t}$ and this gives
\[
2\lambda\int g_{t}'\,d\nu_{t}-\lambda^{2}-\frac{1}{(a+\rho t)^{2}}\left( \int Q' d\nu_{t}\right)^{2}=-\left( \int g_{t}'d\nu_{t}-\lambda \right)^{2}.
\]
Thus, we arrive at  
\begin{equation}\label{e9:tim}
4\partial_{t}\left( \frac{j_{t}}{a+\rho t}\right) \le -\left( \int g_{t}'\,d\nu_{t}-\lambda \right)^{2}\le 0. 
\end{equation}
and then integrating this from $0$ to $1$ gives, 
\[
4\frac{j_{1}}{a+\rho }-4\frac{j_{0}}{a }  \le 0
\]
which combined with \eqref{e5b:1b} becomes
\[
\int g_{1}\,d\mu- \frac{1}{a+\rho}(E_{Q}(\mu)-E_{Q})\le \frac{j_{0}}{a}. 
\]
Now, letting $a\searrow0$  leads to
\begin{equation}\label{e:tmp2}
\int g_{1}\,d\mu- \frac{1}{\rho}(E_{Q}(\mu)-E_{Q}) \le  \int g\,d\mu_{Q}.
\end{equation}
This is the same as
\begin{equation}\label{e5b:10}
\rho\left( \int g_{1}\,d\mu-\int g\,d\mu_{Q}\right)\le E_{Q}(\mu)-E_{Q},
\end{equation}
which, according to Theorem~\ref{t:W} and simple approximations yields
\[
\rho \mathcal{W}_{2}^{2}(\mu,\mu_{Q})\le E_{Q}(\mu)-E_{Q}.  
\]
To see why \eqref{e:tmp2} holds true we take the argument from a version of \cite{L} which is available on page 8 at \url{https://www.math.univ-toulouse.fr/~ledoux/free.pdf}.  For the sake of completeness we sketch the idea.  

We can write for any smooth function $\phi$ on $\T$ and $\delta>0$
\[
j_{Q}(\phi)=\sup_{\nu}\left\{ \int \phi\,d\nu - (E_{Q}(\nu)-E_{Q}) \right\}\le \max\left\{ \int\phi\,d\mu_{Q}+\delta, \sup_{z\in\T}|\phi(z)|-\sup_{\nu\notin\mathcal{A}_{\delta}}(E_{Q}(\nu)-E_{Q})\right\}
\]
where the supremum is taken over all probability measures on $\T$ and $\mathcal{A}_{\delta}=\{\nu\in\mathcal{P}(\T):\int \phi\,d\nu \le \int\phi\,d\mu_{Q}+\delta\}$.  Notice that for any fixed $\delta>0$, $\sup_{\nu\notin\mathcal{A}_{\delta}}(E_{Q}(\nu)-E_{Q})>0$ because of the lower semicontinuity of the functional $E_{Q}(\nu)$ and the uniqueness of the minimizer.  Thus, taking $\phi=ag$, we obtain 
\[
\frac{j_{0}}{a}\le \max\left\{ \int g\,d\mu_{Q}+\delta, \sup_{z\in\T}|\phi(z)|-\frac{1}{a}\sup_{\nu\notin\mathcal{A}_{\delta}}(E_{Q}(\nu)-E_{Q})\right\}
\]
which after letting $a\searrow0$ and then $\delta \searrow0$ gives \eqref{e:tmp2}.  \qedhere

\end{proof}

Notice that in principle we might get an improvement to the transportation inequality because in  \eqref{e9:tim}  we estimate the middle term by $0$, though there is no immediate interpretation of the resulting integral.

\subsection{Transportation, Log-Sobolev and HWI imply Poincar\'e}

The main result is the following.  

\begin{theorem} Take a $C^{3}$ potential $Q$ such that $\inf_{\T}\frac{d\mu_{Q}}{d\alpha}>0$. Then the following hold true: 
\begin{enumerate}
\item $T(\rho)$ implies $P(\rho)$;
\item $LSI(\rho)$ implies $P(\rho)$; 
\item $HWI(\rho)$ implies $P(\rho)$. 
\end{enumerate}
\end{theorem}

\begin{proof}

\begin{enumerate}
\item 
Assume that the transportation inequality $T(\rho)$ holds for some $\rho>0$.  Then using the dual formulation from Theorem~\ref{t:W}  with the notations from there, for any real number $\lambda$ and smooth periodic function $f:\R\to\R$,
\[
\rho \int \mathcal{U}_{1}^{\lambda}fd\mu -\rho\int fd\mu_{Q}\le E_{Q}(\mu)-E_{Q}(\mu_{Q}) 
\]
or equivalently, for any measure $\mu$,
\[
E_{Q}(\mu_{Q})-\rho\int f d\mu_{Q}\le E_{Q-\rho \mathcal{U}_{1}^{\lambda}f}(\mu).
\]
 Minimizing over all measures $\mu$, we get that 
\[
E_{Q}-\rho\int f\,d\mu_{Q}\le E_{Q-\rho \mathcal{U}_{1}^{\lambda}f}.
\]
Now, replacing $f$ by $tf$ with small $t$ and $\lambda$ by $t\lambda$, and the easily checked fact that 
\[
\mathcal{U}^{t\lambda}_{1}(tf)=tf+t^{2}(f'-\lambda)^{2}/4+o(t^{2})
\]
we arrive at
\[
E_{Q}-\rho t\int f\,d\mu_{Q}\le E_{Q-\rho t f -\rho t^{2}(f'-\lambda)^{2}/4+o(t^{2})}.
\]
This combined with \eqref{ep:101}, yields for small $t$, 
\[
E_{Q}-\rho t\int f d\mu_{Q}\le E_{Q}-\rho t\int f\,d\mu_{Q}-t^{2} \left(\frac{\rho^{2}}{2}\langle \mathcal{N}f,f \rangle-\frac{\rho}{4}\int(f'-\lambda)^{2}d\mu_{Q} \right)+o(t^{2}).
\]
In turn this implies
\[
2\rho\langle \mathcal{N}f,f \rangle\le \int (f'-\lambda)^{2}d\mu_{Q}.
\]  
Finally, minimizing over $\lambda$ results with 
\[
2\rho\langle \mathcal{N}f,f \rangle\le \int (f')^{2}d\mu_{Q}-\left(\int f'd\mu_{Q}\right)^{2}
\]
which is precisely $P(\rho)$.  

\item From the above implications, we know that Log-Sobolev implies the transportation, thus in particular also implies the Poincar\'e.  However, we provide a direct proof here which is relatively simple and shows that the linearization of Log-Sobolev is also the Poincar\'e.   

To see this, take a smooth function $f$ such that $\int f\,d\alpha=0$ and consider now the measure $\nu_{t}=\mu_{Q}-t\mathcal{N}f\alpha$.  Notice that for small $t$ this is well defined, in other words, a probability measure.  In fact, we learn from \eqref{e:d} that
\[
\nu_{t}=(1-\mathcal{N}(Q+tf))\alpha.
\]

 Now apply $LSI(\rho)$ to obtain that 
\[
4\rho(E_{Q}(\nu_{t})-E_{Q})\le \int (H\nu_{t}-Q')^{2}(1-\mathcal{N}(Q+tf))d\alpha- \langle Q',\mathcal{N}(Q+tf) \rangle^{2}. 
\]
Notice now that from Proposition~\ref{p:1},
\begin{align*}
E_{Q}(\nu_{t})&=E_{Q}-t\langle \mathcal{N}f,Q \rangle+2t\iint \log|z-w|\mathcal{N}f\,\alpha(dw)\mu_{Q}(dz)-t^{2}\iint \log|z-w|\mathcal{N}f(z)\mathcal{N}f(w)\alpha(dw)\alpha(dz) \\ 
&= E_{Q}-t\langle \mathcal{N}f,Q \rangle-t\int \mathcal{E}\mathcal{N}f(z)\,d\mu_{Q}(dz)+\frac{t^{2}}{2}\langle \mathcal{E}\mathcal{N}f,\mathcal{N}f \rangle\\
&= E_{Q}+\frac{t^{2}}{2}\langle \mathcal{N}f,f 
\rangle
\end{align*}
where the cancellation of the  coefficient of $t$ is again a consequence of Proposition~\ref{p:1}.   Furthermore, since $H(f\alpha)=(\mathcal{E}f)'=\mathcal{E}f'$ and  $H(\nu_{t})=H(\mu_{Q})-tH(\mathcal{N}f\alpha)=Q'+t(\mathcal{E}\mathcal{N}f)'=Q'+tf'$ combined with the fact that $\mathcal{N}$ commutes with the derivative and the above computation, we arrive at
\[
2\rho t^{2}\langle\mathcal{N}f,f \rangle\le t^{2}\int (f')^{2}\,d\mu_{Q}-t^{2}\left( \int f' d\mu_{Q} \right)^{2}+o(t^{2}), 
\]
which implies the Poincar\'e inequality  $P(\rho)$ as set in \eqref{e1:1}. 

\item Since $HWI(\rho)$ implies $LSI(\rho)$, the proof follows.  
\end{enumerate}

\end{proof}

\begin{remark}
For the case of $Q=0$, we proved $T(1/4)$, $LSI(1/4)$ and $HWI(1/4)$.  All of these in turn imply $P(1/4)$,  which is a weaker result than the one provided by  Theorem~\ref{t:p}.
\end{remark}

\section{Potential independent inequalities} \label{S:10}

In this section we introduce some versions of the functional inequalities which are independent of the potential.  

We start with the transportation inequality.   The idea is the following.  Assume that $Q$ is an arbitrary potential on $\T$ and $\mu_{Q}$ is the equilibrium measure associated to it.  Then, by the variational characterization of the equilibrium measure $\mu_{Q}$, we know that for some constant $C$,
\[
Q(z)\ge2\int \log|z-w|\mu_{Q}(dw)+C
\]
with equality on the support of $\mu_{Q}$.  Consequently for any other measure $\mu$, we have 
\[
\begin{split}
E_{Q}(\mu)-E_{Q}&=\int Q d\mu -\iint \log|z-w|\mu(dz)\mu(dw)-\int Qd\mu_{Q}+\iint \log|z-w|\mu_{Q}(dz)\mu_{Q}(dw) \\
&\ge -\iint \log|z-w|(\mu-\mu_{Q})(dz)(\mu-\mu_{Q})(dw).
\end{split}
\]
Thus it makes sense to define the following version of the relative entropy, one which is independent of the potential.  For any two measures on $\T$ with finite logarithmic energy we set
\begin{equation}\label{e10:H}
\mathcal{H}(\mu,\nu)=-\iint \log|z-w|(\mu-\nu)(dz)(\mu-\nu)(dw). 
\end{equation}

The first result of this section is the following transportation inequality.  

\begin{theorem}
For any $\mu,\nu\in\mathcal{P}(\T)$, 
\begin{equation}\label{e:pi:1}
\mathcal{W}_{1}^{2}(\mu,\nu)\le 2\mathcal{H}(\mu,\nu).
\end{equation}
The equality is not attained.  

In particular, for any potential $Q$ on $\T$, we have 
\[
\mathcal{W}_{1}^{2}(\mu,\mu_{Q})\le 2(E_{Q}(\mu)-E_{Q}).
\]
\end{theorem}

\begin{proof}
By careful approximations, it suffices to assume that $\mu-\nu=\phi\, \alpha$ for some smooth function $\phi:\T\to\R$.  From \eqref{e:t21} we have 
\[
\mathcal{W}_{1}(\mu,\nu)=\sup_{\lambda,g} \left\{ \int g\,d(\mu-\nu): |g(x)-g(y)-\lambda(x-y)|\le |x-y| \right\},
\]
where the supremum is taken over all $\lambda\in\R$ and periodic functions $g$.  

Now, for any smooth function $g$ such that $|g'-\lambda|\le 1$ ($\alpha$-almost surely),  \eqref{e:pe:0} and the fact that the derivatives commute with $\mathcal{E}$ combined with \eqref{ep:10:e} of Proposition~\ref{p:1}, we obtain that  
\[
\begin{split}
\int g\phi\,d\alpha=\langle g,\phi \rangle&= \langle g',\mathcal{E}^{2}\phi'  \rangle=\langle g'-\lambda,\mathcal{E}^{2}\phi'  \rangle\le \left(\int |g'-\lambda|^{2}d\alpha\right)^{1/2}\left( \int |\mathcal{E}^{2}\phi'|^{2}d\alpha\right)^{1/2} \\ 
&\le \langle \mathcal{E}^{2}\phi' ,\mathcal{E}^{2}\phi' \rangle^{1/2} =\langle \mathcal{E}^{2}\phi,\phi\rangle^{1/2}  \\
& \le \langle \mathcal{E}\phi ,\phi \rangle^{1/2} =(2\mathcal{H}(\mu,\nu))^{1/2}.
\end{split}
\]

Tracing back the inequalities we get that equality is attained for the case of $|g'-\lambda|\equiv1$ ($\alpha$-almost surely), $\mathcal{E}^{2}\phi'=c (g'-\lambda)$ for some constant $c\ge0$ and $\phi$ must be in the eigenspace of eigenvalue $1$ for $\mathcal{E}$, which means $\phi(x)=a\cos(x)+b\sin(x)$ where we exclude the trivial case $a=b=0$.   From this we deduce that $c(g'-\lambda)=\mathcal{E}^{2}\phi'=\phi'$ which is impossible unless $a=b=0$ which we precluded.  Thus \eqref{e:pi:1} does not have an equality case.   The rest follows.  \qedhere
\end{proof}

Despite the fact the equality is not attained, the inequality might still be sharp, though an approximate sequence for which this happens is not clear to the author.

To move next with the Log-Sobolev inequality we need to find a replacement of the free information inequality, something similar to $\mathcal{H}(\mu,\nu)$ from \eqref{e10:H}.  To do this, just notice that the free relative Fisher information is
\[
I_{Q}(\mu)=\int(H\mu-Q')^{2}d\mu-\left(\int Q' d\mu \right)^{2}
\]
and if we recall that (at least formally) we have that 
\[
Q'=H\mu_{Q},
\]
then we can rewrite the above in the from
\[
I_{Q}(\mu)=\int(H(\mu-\mu_{Q}))^{2}d\mu-\left(\int H(\mu-\mu_{Q}) d\mu \right)^{2}
\]
where we used the fact that $\int H\mu d\mu=0$.  If we want to make this quantity into one which is potential independent and also symmetric in $\mu$ and $\mu_{Q}$ one natural suggestion is to replace the integration from the integration with respect to the measure $\mu$ into an integration with respect to another measure which is independent of both $\mu$ and $\mu_{Q}$.  Such a natural measure we can pick is the Haar measure  $\alpha$.  Moreover, for measures of the form $d\mu=\phi d\alpha$ with $\phi$ a smooth function we can argue that 
\[
H\mu=(\mathcal{E}\phi)'. 
\]
Furthermore, since for $n\ne 0$, $\mathcal{E}z^{n}=\frac{1}{|n|}z^{n}$, we get that $(H\mu)(z)=i\mathrm{sign}(n)z^{n}$.  This shows that for any smooth, real valued function $\phi$, 
\[
\int (H\phi)^{2}d\alpha=\int \phi^{2}d\alpha-\left(\int\phi d\alpha\right)^{2}.  
\]
Consequently we define for two probability measures $\mu,\nu$ on $\T$
\begin{equation}\label{e10:I}
\mathcal{I}(\mu,\nu)=\begin{cases}
\int \left(\frac{d\mu}{d\alpha}-\frac{d\nu}{d\alpha} \right)^{2}d\alpha & \text{ if } \frac{d\mu}{d\alpha},\frac{d\nu}{d\alpha}\in L^{2}(\alpha),  \\
+\infty & \text{ otherwise}.
\end{cases}
\end{equation}
Notice that we do not subtract the natural quantity $\int \left(\frac{d\mu}{d\alpha}-\frac{d\nu}{d\alpha}\right)d\alpha$ since this is already $0$ in the case both $\mu$ and $\nu$ are absolutely continuous with densities in $L^{2}(\alpha)$.   

With this definition we have the Log-Sobolev inequality as follows. 
\begin{theorem}
For any measures $\mu,\nu$, we have that 
\begin{equation}\label{e10:LSI}
\mathcal{H}(\mu,\nu)\le \frac{1}{2}\mathcal{I}(\mu,\nu). 
\end{equation}
Equality is attained for cases of $\mu(dz)-\nu(dz)=(a\Re(z)+b\Im(z))\alpha(dz)$ for some $a,b\in\R$.   
\end{theorem}

\begin{proof}  By standard approximation results, we can reduce the proof to the case of $\mu-\nu =\phi d\alpha$ where $\phi$ is a smooth function.  Using the definition of $\mathcal{E}$, what we need to show is that 
\[
\langle \mathcal{E}\phi ,\phi\rangle \le \langle \phi,\phi\rangle
\]
which immediately follows from the fact that the spectrum of $\mathcal{E}$ is $0,1,1/2,1/3,\dots$.   Equality is also clear.
\end{proof}

In a similar vein we can also prove the HWI inequality which takes the following form.  

\begin{theorem}
For any measures $\mu,\nu$ on $\T$, 
\[
\mathcal{H}(\mu,\nu)\le \sqrt{\mathcal{I}(\mu,\nu)}\mathcal{W}_{1}(\mu,\nu)-\frac{1}{2}\mathcal{W}_{1}^{2}(\mu,\nu).
\]
\end{theorem}

\begin{proof}  For the proof, we just need to show it for measures $\mu,\nu$ which are absolutely continuous with respect to $\alpha$ and in addition, the densities are also smooth.  By a simple scaling argument it suffices to do it for $\phi$ such that $|\mathcal{E}^{2}\phi'|\le1$.  

The HWI inequality becomes equivalent to 
\[
\sqrt{\langle \phi,\phi \rangle}-\sqrt{\langle \phi,\phi \rangle - \langle \mathcal{E}\phi,\phi \rangle}\le \sup_{\lambda,g} \left\{ \int g\phi\, d\alpha: |g(x)-g(y)-\lambda(x-y)|\le |x-y| \right\} \le \sqrt{\langle \phi,\phi \rangle}+\sqrt{\langle \phi,\phi \rangle - \langle \mathcal{E}\phi,\phi \rangle}.
\]
The right hand side inequality is a simple consequence of the transportation and Log-Sobolev inequalities.   To show the left hand inequality it suffices to take $\lambda=0$ and $g=\mathcal{E}^{2}\phi$.  In particular, the supremum is bounded below by $\langle \mathcal{E}^{2}\phi,\phi \rangle$, thus all we need is to show that 
\[
\sqrt{\langle \phi,\phi \rangle}-\sqrt{\langle \phi,\phi \rangle - \langle \mathcal{E}\phi,\phi \rangle}\le  \sqrt{\langle \mathcal{E}^{2}\phi,\phi \rangle}.
\] 
Writing $\phi(z)=\sum_{n\ne 0}a_{n}z^{n}$, our task is now becomes
\[
\sqrt{\sum_{n\ne0}a_{n}^{2}}\le \sqrt{\sum_{n\ne0}(1-1/|n|)a_{n}^{2}}+ \sqrt{\sum_{n\ne0}a_{n}^{2}/n^{2}}.
\]
Squaring and simplifying reduces it to  
\[
\sum_{n\ne0}\frac{1}{|n|}\left(1-\frac{1}{|n|}\right)a_{n}^{2}\le 4\left(\sum_{n\ne0}\frac{1}{n^{2}}a_{n}^{2}\right)\left(\sum_{n\ne0}\left(1-\frac{1}{|n|}\right)a_{n}^{2}\right) 
\]
which follows simply from Cauchy inequality and the fact that $(1-1/|n|)\le 1$. \qedhere 

\end{proof}

\section{Back to the classical case of functional inequalities}\label{S:11}

The free case on the circle opens up an interesting question in the frame of classical functional inequalities.  To set up the scene, assume that $M$ is a Riemannian manifold and $G$ a Lie group acting on $M$ with $\mu$ a probability measure on $M$, say of the form $\mu(dx)=e^{-V(x)}dx$, where $dx$ is the standard volume measure on $M$.   Assume that $V$ satisfies the typical Bakry-Emery condition 
\[
Hess V+ Ric\ge \rho I,
\]
for some positive $\rho$.  Then all the classical inequalities, as for instance, the transportation, Log-Sobolev, HWI and the weakest of them, the Poincar\'e inequality as described in the introduction all hold true.   

In this setup we can ask the following natural question.  

\emph{Are there versions of these inequalities which take into account how far the measure $\mu$ is from being invariant under the group action?}  

For example we can ask about the simplest of them, namely the Poincar\'e inequality.  One possibility is the following: 
\begin{equation}
\rho\Var_{\mu}(\phi)\le \int |\nabla \phi|^{2}d\mu - \left|\int \nabla_{G} \phi d\mu \right|_{\mathfrak{g}}^{2}
\end{equation}
where $(\nabla_{G}\phi)(x)$ is the element in the Lie algebra $\mathfrak{g}$ such that 
\[
\langle (\nabla_{G}\phi)(x),g \rangle=\frac{d}{dt}\big|_{t=0}\phi(\exp(tg)x)
\]
where $\exp(g)$ is the exponential map from $\mathfrak{g}$ into $G$.   We need to endow here $\mathfrak{g}$ with an appropriate metric which then is extended to an invariant metric on $G$.    Obviously if the measure $\mu$ is invariant under the action of the group, then $\nabla_{G}\phi=0$ and thus we basically fall back into the classical case.  

It would be interesting to see the transportation inequality and similarly Log-Sobolev in this framework.  In particular, it would be nice to see a form of the Wasserstein distance which incorporates the action of the group $G$ on the manifold $M$.

\section{Thanks}

I owe a lot to the anonymous reviewer of the submitted paper for careful reading and comments which substantially improved this paper.  

I want to thank Michael Loss for useful discussions and the idea of using Fourier expansion in the proof of Proposition \ref{p:6}. 

Many thanks to Haussdorf Institute in Bonn for its hospitality, great and stimulating environment during the trimester of mass transportation in 2015 where part of this paper was expanded and improved.

\providecommand{\bysame}{\leavevmode\hbox to3em{\hrulefill}\thinspace}
\providecommand{\MR}{\relax\ifhmode\unskip\space\fi MR }
\providecommand{\MRhref}[2]{%
  \href{http://www.ams.org/mathscinet-getitem?mr=#1}{#2}
}
\providecommand{\href}[2]{#2}

\end{document}